\documentclass[11pt]{article}
\usepackage[a4paper,margin=1in]{geometry}

\usepackage{amsmath,amssymb,amsthm}
\usepackage{graphicx}
\usepackage{booktabs}
\usepackage{enumitem}
\usepackage{pdflscape}
\usepackage{xcolor}
\usepackage{algorithm}
\usepackage{algpseudocode}
\usepackage{calrsfs}
\usepackage{url}

\graphicspath{{fig/}}

\DeclareMathOperator*{\argmin}{arg\,min}

\newcommand{\Prob}{\ensuremath{\mathbb{P}}}
\newcommand{\Ex}{\ensuremath{\mathbb{E}}}

\newcommand{\sgn}{\mathop{\mathrm{sgn}}}

\newcommand{\rev}[1]{#1}

\newtheorem{lemma}{Lemma}

\newtheorem{proposition}{Proposition}
\newtheorem{theorem}{Theorem}
\newtheorem{corollary}{Corollary}

\newtheorem{remark}{Remark}

\providecommand{\bibcommenthead}{}
\newenvironment{barticle}{}{}
\newenvironment{bbook}{}{}
\newenvironment{bchapter}{}{}
\newenvironment{botherref}{}{}

\usepackage[authoryear,round]{natbib}
\usepackage{hyperref}
\hypersetup{
  colorlinks=true,
  linkcolor=blue,
  citecolor=blue,
  urlcolor=blue
}

\begin{document}

\title{A belief-state restless bandit model for treatment adherence: Whittle indexability via partial conservation laws}

\author{%
Jos\'e Ni\~no-Mora \quad and \quad \'Angel Pellitero Garc\'{\i}a\\[0.5ex]
{\small Departamento de Estad\'istica, Universidad Carlos III de Madrid}\\
{\small Getafe (Madrid), Spain}\\[0.5ex]
{\small \texttt{jose.nino@uc3m.es} \quad \texttt{angelpelli13@gmail.com}}
}

\date{}
\maketitle

\begin{abstract}
We study clinically motivated capacity-constrained treatment-adherence outreach through a belief-state restless multi-armed bandit model, in which each patient is a partially observed two-state Markov decision process and interventions induce reset-type belief dynamics. For the discounted criterion, partial conservation law (PCL)-based conditions are used to establish single-patient threshold-policy optimality and indexability (threshold-indexability) and yield a closed-form Whittle index, threshold performance metrics, and an explicit optimal threshold map. We also prove a single-patient long-run average analogue on the invariant belief core and obtain an explicit average-criterion Whittle index. For the multi-patient model, the PCL-derived formulas give an analytic Lagrangian relaxation, efficient dual bounds, and computable Lagrangian index benchmark policies, including a forced-capacity variant. We analyze how the Whittle index depends on lapse and spontaneous-recovery probabilities. In large-scale experiments with two-type, three-type, and jittered finite-mixture populations, the Whittle and forced-capacity Lagrangian index policies are the strongest performers, while myopic prioritization can be substantially worse under tight capacity.
\end{abstract}

\medskip

\section{Introduction}
\label{s:intro}

Treatment \emph{adherence}---following a recommended regimen in dose, timing, and duration---is a major determinant of health outcomes; higher adherence is associated with improved clinical control and fewer adverse events across chronic conditions \citep{Peng2025AHTadherenceMortality,Chongmelaxme2020AsthmaAdherence,Nguyen2024InsulinAdherenceA1c}. Because health services can only deliver a limited number of outreach actions, such as calls or visits, this paper develops a scalable policy for prioritizing adherence outreach over time under capacity constraints. We focus on index policies and computable performance benchmarks, including Whittle indices and Lagrangian relaxation bounds.

Each patient is modeled as a two-state adherent/nonadherent \emph{partially observed Markov decision process} (POMDP) \citep{krishn25}. The adherence state is observed only when outreach occurs, and a high-touch intervention, or activation, produces an immediate reset-type adherence effect before subsequent Markov evolution. The controller acts on continuous beliefs and chooses in each period which \(M\) of \(N\) patients to activate, yielding a belief-state \emph{restless multi-armed bandit problem} (RMABP) \citep{whit88,nmmath23}. Whittle's index policy is attractive in this setting, but it requires proving indexability and computing the index in belief space.

Our proof uses the partial conservation laws (PCLs) verification theorem in \citet{nmmor20} for real-state models, which extends earlier work for discrete-state models in \citet{nmaap01,nmmp02,nmmor06}. Satisfaction of the \emph{PCL-indexability conditions}, a set of conditions on  reward and work metrics under threshold policies, implies simultaneously \emph{both} optimality of threshold policies \emph{and} Whittle indexability, which we call \emph{threshold-indexability}, avoiding sequential proofs of threshold optimality, threshold monotonicity, and index inversion. Thus PCL-indexability is a verification device to conclude threshold-indexability with an explicit closed-form Whittle index.

The closed-form performance metrics further give a computable Lagrangian relaxation bound. The single-patient Lagrangian is piecewise affine in the intervention price, the multi-patient bound is computed by bisection with \(O(N)\) work per iteration, and the marginal metrics define two Lagrangian index benchmarks: a standard positive-index rule and a forced-capacity rule that always activates the \(M\) largest indices.

Although our focus is on the discounted reward criterion, we also prove a corresponding threshold-indexability result under the long-run average reward criterion. 

\paragraph{Contributions}
We make six contributions: (i) a belief-state RMABP model for capacity-constrained adherence outreach with reset-type interventions; (ii) a PCL-based proof of discounted threshold-indexability and a closed-form Whittle index; (iii) a corresponding PCL-based long-run average threshold-indexability result on the invariant belief core; (iv) a compact illustration of the PCL route as an alternative to direct value-function and threshold-monotonicity proofs; (v) a closed-form single-patient Lagrangian, an \(O\!\bigl(N\log(1/\varepsilon)\bigr)\)-time bisection scheme for an \(\varepsilon\)-accurate dual bound, and two Lagrangian index benchmarks; and (vi) parameter-comparative statics and large-scale comparisons with myopic, round-robin, Lagrangian index, forced-capacity Lagrangian index, and passive policies. Across the reported experiments, Whittle and forced-capacity Lagrangian index policies are the strongest performers.

\subsection{Related work: belief-state adherence bandits and dual bounds}
\label{s:rwbsabdb}

RMABP models have been used to prioritize adherence interventions in large-scale deployments \citep{mateetal22}, where each beneficiary is modeled through a fully observed two-state engaging/non-engaging MDP\@. Belief-state adherence bandits with \emph{collapsing} or state-revealing observations under activation have also been studied \citep{Mate2020NeurIPS,Mate2021RiskAwareRMAB}. Those models and ours use Bayesian filtering in belief space; here, however, the belief update is deterministic---affine under passivity and reset under activation---conditional on the initial belief.

Lagrangian dual bounds are classical benchmarks for restless bandits \citep{whit88,wewe90}. In belief-state RMABs, the relaxation decouples the problem into single-arm POMDPs with an intervention charge, often requiring numerical value-function approximation, for example by point-based value iteration or rollout \citep{kazaetal19}. In contrast, our PCL analysis yields an explicit Whittle index, an analytic dual bound, and the marginal quantities used by the Lagrangian index benchmarks.

A related structure-exploiting analysis is in \citet{liuZhao10}'s belief-state channel-access model. They derive closed-form single-arm value functions and exploit a discounted passive-time staircase to search for the minimizing subsidy, obtaining an \(\varepsilon\)-accurate bound in \(O\!\bigl(N^2\log N\bigr)\) time and exactness outside a gray area. In our model, each arm's Lagrangian and resource terms are evaluated in constant time for a given multiplier, so each bisection step costs \(O(N)\) and the dual bound is computed in \(O\!\bigl(N\log(1/\varepsilon)\bigr)\) time.

The remainder of the paper is organized as follows. Section~\ref{s:model} gives the model and RMABP formulation. Section~\ref{s:iwip} introduces the Lagrangian relaxation, dual bound, and Whittle policy. Sections~\ref{s:pclbvtri}--\ref{s:avpclic} verify discounted PCL-indexability and derive the closed-form Whittle index and optimal threshold map. Section~\ref{s:dual-bound} derives the analytic Lagrangian bound and bisection scheme. Section~\ref{s:dmpimp} analyzes parameter dependence. Section~\ref{s:iatac} proves indexability under the long-run average criterion. Section~\ref{s:numerics} reports the numerical study, and Section~\ref{s:concl} concludes. Proofs and additional numerical summaries are in the Online Supplement.

\section{Model description and RMABP formulation}
\label{s:model}

We consider a cohort of $N$ patients monitored over discrete time
$t=0,1,\ldots$. In each period, at most $M<N$ patients can receive an
adherence-boosting intervention. Patient $n$ has a latent adherence state $Y_n(t)\in\{0,1\}$, with $Y_n(t)=1$ denoting adherence and $Y_n(t)=0$ nonadherence, and receives an action $A_n(t)\in\{0,1\}$, where $A_n(t)=1$ indicates an \emph{intervention} (the \emph{active} action in restless bandit terminology) and $A_n(t)=0$ indicates \emph{no intervention} (the \emph{passive} action), subject to the capacity constraint
\begin{equation}
\label{eq:resconst}
\sum_{n=1}^N A_n(t)\leqslant M, \quad t = 0, 1, \ldots.
\end{equation}
Conditional on $A_n(t)=a$, the latent state evolves as a controlled Markov
chain with the following transition matrices, where $0<p_n,q_n<1$:
\begin{equation}
\label{eq:Pn01}
\mathbf{P}_n^0=
\bordermatrix{~ & 0 & 1 \cr
               0 & 1-q_n & q_n \cr
               1 & p_n   & 1-p_n \cr},
\quad
\mathbf{P}_n^1=
\bordermatrix{~ & 0 & 1 \cr
               0 & p_n   & 1-p_n \cr
               1 & p_n   & 1-p_n \cr},
\end{equation}
Here $p_n$ is the \emph{lapse probability} from adherence and
$q_n$ is the \emph{spontaneous recovery probability} under no intervention. Under activation, the transition is of reset type: the next state is adherent with probability \(1-p_n\), regardless of the current state. For a currently nonadherent patient, passivity would instead yield spontaneous recovery with probability \(q_n\). We therefore assume a beneficial intervention effect,
\begin{equation}
\label{eq:positautocor}
1-p_n>q_n,
\end{equation}
which says that outreach improves the next-period adherence probability relative to unaided recovery. This is natural in the adherence-outreach setting, where spontaneous recovery without contact is not expected to exceed post-intervention adherence. Equivalently, under passivity the latent state has positive persistence parameter
\(\rho_n \triangleq 1-p_n-q_n>0\), so the passive belief update is increasing.

The latent state is observed only when an intervention is delivered. We
therefore use the \emph{belief state}
\begin{equation}
\label{eq:Xnt}
X_n(t)\triangleq \Prob\{Y_n(t)=0\mid \mathcal{F}_t\}\in[0,1],
\end{equation}
the posterior probability of nonadherence given the available information
$\mathcal{F}_t$. In this model the belief dynamics are deterministic:
\[
X_n(t+1)=
\begin{cases}
p_n+\rho_n X_n(t), & A_n(t)=0,\\
p_n,               & A_n(t)=1.
\end{cases}
\]

Throughout, we refer to each arm/project as a patient; we use the terms patient, project, and arm interchangeably.

Let $x\in[0,1]$ denote the belief of nonadherence and let $a\in\{0,1\}$ be the
action. Rewards accrue from adherence: under passivity the expected one-step
reward is $r_n(1~-~x)$, while under activation the reset yields reward $r_n$.
Thus
\[
R_n(x,a)\triangleq r_n\bigl[(1-x)+ax\bigr]
=
\begin{cases}
r_n, & a=1,\\
r_n(1-x), & a=0.
\end{cases}
\]

A policy $\pi\in\Pi(M)$ is any nonanticipative (possibly randomized) rule that
satisfies \eqref{eq:resconst}. Given an initial belief vector
$\mathbf{x}^0=(x_1^0,\ldots,x_N^0)$ and discount factor $\beta\in(0,1)$, and writing as $\Ex_{\mathbf{x}^0}^{\pi}$ the expectation under $\pi$ starting from $\mathbf{x}^0$, the
RMABP of concern is
\begin{equation}
\label{eq:rmabp}
V^*(\mathbf{x}^0)\triangleq
\sup_{\pi\in\Pi(M)}
\Ex_{\mathbf{x}^0}^{\pi}\!\left[
\sum_{t=0}^{\infty}\sum_{n=1}^{N} R_n\bigl(X_n(t),A_n(t)\bigr)\beta^{t}
\right].
\end{equation}

\begin{remark}[Deterministic belief dynamics]
\label{rem:deterministic-belief}
Given $\mathbf{x}^0$, the belief trajectory is deterministic for any
deterministic policy. We retain the
expectation notation to allow randomized policies and to match RMABP
conventions.
\end{remark}

\begin{remark}[Per-intervention costs]\label{rem:costs}
If interventions incur costs $c_n\geqslant 0$, replace $R_n(x,a)$ by
$\widetilde{R}_n(x,a)=R_n(x,a)-c_n a$. In the Lagrangian formulation below this
shifts the effective intervention price from $\lambda$ to $c_n+\lambda$, yielding
the cost-adjusted index $m_n(x)-c_n$. We focus on the case $c_n\equiv0$ for clarity.
\end{remark}

\rev{
For ease of reference, Table~\ref{tab:notation} summarizes the main recurring notation. Patient-specific quantities carry the subscript \(n\); in the single-patient analysis this subscript is suppressed.
}

\begin{table}[t]
\centering
\caption{Main recurring notation.}
\label{tab:notation}
\small
\setlength{\tabcolsep}{4pt}
\renewcommand{\arraystretch}{1.12}
\begin{tabular}{@{}p{0.31\textwidth}p{0.63\textwidth}@{}}
\toprule
Symbol & Meaning \\
\midrule
\(Y_n(t),X_n(t),A_n(t),R_n(x,a)\) & Latent state, nonadherence belief, intervention decision, and one-period reward. \\
\(p_n,q_n,\rho_n\) & Lapse probability, spontaneous-recovery probability, and passive persistence \(\rho_n=1-p_n-q_n\). \\
\(h(x),z_\infty,z_t,\tau(x,z)\) & Passive belief update, passive fixed point, threshold breakpoints, and first threshold-upcrossing time. \\
\(F,G,f,g\) & Threshold reward/work metrics and their marginal reward/work counterparts. \\
\(m_n,w_n,z_n^*(\lambda)\) & MP index, Whittle index, and optimal threshold at price \(\lambda\). \\
\(\lambda^*,D,d_n^*(x)\) & Dual-optimal price, Lagrangian dual bound, and Lagrangian benchmark index. \\
\(\bar V^\pi,\bar D,\gamma^\pi\) & Normalized policy value, normalized dual bound, and relative Lagrangian gap. \\
\(\mathsf X^{\mathrm{inv}},m^{\mathrm{avg}}\) & Invariant core and average-criterion Whittle index. \\
\bottomrule
\end{tabular}
\end{table}

\section{Lagrangian bound, indexability, and Whittle index}
\label{s:iwip}
Problem~\eqref{eq:rmabp} is an RMABP. Since exact dynamic programming is generally intractable, we use Whittle's index policy~\citep{whit88} and Lagrangian dual bounds. Relax the hard capacity constraint~\eqref{eq:resconst} to the discounted aggregate constraint
\begin{equation}
\label{eq:relaxp}
\Ex_{\mathbf{x}^0}^{\pi}\!\left[\sum_{t=0}^{\infty}\sum_{n=1}^{N} A_n(t)\beta^{t}\right]
\leqslant \frac{M}{1-\beta},
\end{equation}
and let \(\widehat{V}^*(\mathbf{x}^0)\) be the relaxed value. Then \(V^*(\mathbf{x}^0)\le\widehat{V}^*(\mathbf{x}^0)\). For a multiplier \(\lambda\geqslant 0\), define
\begin{equation}
\label{eq:lagrelax}
L(\mathbf{x}^0;\lambda) \triangleq \sup_{\pi\in \Pi}
\Ex_{\mathbf{x}^0}^{\pi}\!\left[
\sum_{t=0}^{\infty}\sum_{n=1}^{N} \bigl(R_n(X_n(t),A_n(t))-\lambda A_n(t)\bigr)\beta^{t}
\right]
+\frac{M}{1-\beta}\lambda,
\end{equation}
where \(\Pi\) is the unrestricted nonanticipative policy class. Weak duality gives \(\widehat{V}^*(\mathbf{x}^0)\leqslant L(\mathbf{x}^0;\lambda)\), hence the multi-patient Lagrangian dual bound
\begin{equation}
\label{eq:dualbound}
D(\mathbf{x}^0)\triangleq\inf_{\lambda\geqslant 0}L(\mathbf{x}^0;\lambda)
\end{equation}
upper-bounds the original RMABP value. When the infimum is attained, write a minimizer as \(\lambda^*\).

The relaxation decouples into \(N\) single-patient \(\lambda\)-price subproblems,
\begin{equation}
\label{eq:lambdasubp}
L_n(x_n^0;\lambda)\triangleq
\sup_{\pi_n\in \Pi_n}
\mathbb{E}_{x_n^0}^{\pi_n}\!\left[
\sum_{t=0}^{\infty}\bigl(R_n(X_n(t),A_n(t))-\lambda A_n(t)\bigr)\beta^{t}
\right],
\end{equation}
so \(L(\mathbf{x}^0;\lambda)=\sum_{n=1}^N L_n(x_n^0;\lambda)+M\lambda/(1-\beta)\). Each \(L_n\), and hence \(L\), is convex in \(\lambda\).

We say that patient \(n\)'s subproblem is \emph{indexable} if there is an index map \(w_n\colon [0,1]~\to~\mathbb R\) such that action 1 is optimal iff \(w_n(x_n)\geqslant \lambda\) and action 0 is optimal iff \(w_n(x_n)\le\lambda\). It is \emph{threshold-indexable} if, in addition, each \(\lambda\)-price problem has an optimal threshold policy that activates when the belief exceeds a threshold. We call \(w_n\) the \emph{Whittle index}. Under indexability, Whittle's policy evaluates \(w_n(X_n(t))\) each period and activates up to \(M\) patients with the largest nonnegative indices.

The standard route to indexability in ordered state spaces, as illustrated in \cite{liuZhao10}, is to first prove threshold optimality for each \(\lambda\), then prove monotonicity of the optimal threshold in \(\lambda\), and finally invert the threshold--price relation. We instead use the PCL verification theorem in \citet{nmmor20}: after checking structural identities for threshold reward and work metrics, the theorem yields threshold optimality, indexability, and an explicit Whittle index formula.

\rev{
\begin{remark}[Lagrangian index benchmarks]
\label{rem:relax-index-heuristic}
The same relaxation also defines index benchmarks that do not require Whittle indexability. Given a dual minimizer \(\lambda^*\), let
\begin{equation}
\label{eq:lagindex}
\begin{aligned}
d_n^*(x)\triangleq{}& R_n(x,1)-\lambda^*
+\beta\,\Ex\!\left[L_n(X_n(t+1);\lambda^*)\mid X_n(t)=x,A_n(t)=1\right]\\
&{}-R_n(x,0)
-\beta\,\Ex\!\left[L_n(X_n(t+1);\lambda^*)\mid X_n(t)=x,A_n(t)=0\right].
\end{aligned}
\end{equation}
The standard Lagrangian index policy activates positive \(d_n^*(X_n(t))\)'s up to capacity; the forced-capacity version activates the \(M\) largest \(d_n^*(X_n(t))\)'s regardless of sign. This is a structured value-function-approximation benchmark: the one-step comparison uses exact single-patient relaxed values at \(\lambda^*\), not the value of a feasible base policy. It is related to first-order relaxation and LP/Lagrangian index heuristics \citep{benm00,brownsmith20,gastetal24}. Unlike Whittle's policy, it depends on the initial population through \(\lambda^*\). Below we express \(d_n^*\) in terms of the same PCL marginal metrics used for the Whittle index and dual bound.
\end{remark}
}

\section{PCL-based verification theorem for threshold-indexability}
\label{s:pclbvtri}
We now fix a single patient and suppress its label. For a nonanticipative policy \(\pi\), let
\[
F(x,\pi)=\mathbb{E}^{\pi}_{x}\!\left[\sum_{t\geqslant 0} R(X(t),A(t))\beta^{t}\right],
\qquad
G(x,\pi)=\mathbb{E}^{\pi}_{x}\!\left[\sum_{t\geqslant 0} A(t)\beta^{t}\right]
\]
be discounted reward and work. The \(\lambda\)-price performance is \(\mathcal{L}(x,\pi;\lambda)=F(x,\pi)-\lambda G(x,\pi)\), and the optimal value is
\begin{equation}
\label{eq:lppg}
L(x;\lambda)=\sup_{\pi\in\Pi}\mathcal{L}(x,\pi;\lambda).
\end{equation}

The \emph{\(z\)-threshold policy}, for \(z\in\overline{\mathbb R}=\mathbb R\cup\{-\infty,\infty\}\), activates when \(x>z\) and is passive otherwise; write \(F(x,z)\), \(G(x,z)\), and \(\mathcal L(x,z;\lambda)\) for its metrics. The family is \emph{thresholdable} if some threshold map \(z^*(\lambda)\) satisfies, for all initial states \(x\),
\begin{equation}
\label{eq:Lxlambdazstarl}
L(x;\lambda)=\sup_{z\in\overline{\mathbb R}}\mathcal L(x,z;\lambda)=\mathcal L(x,z^*(\lambda);\lambda).
\end{equation}
\emph{Threshold-indexability} means both optimality of threshold policies (\emph{thresholdability}) and Whittle indexability.

For \(a\in\{0,1\}\), let \(\langle a,z\rangle\) take action \(a\) at time zero and then follow threshold \(z\). Define marginal reward and work by
\[
f(x,z)=F(x,\langle1,z\rangle)-F(x,\langle0,z\rangle),
\qquad
g(x,z)=G(x,\langle1,z\rangle)-G(x,\langle0,z\rangle).
\]
When \(g(x,z)>0\), define the MP metric \(m(x,z)=f(x,z)/g(x,z)\) and the diagonal MP index \(m(x)=m(x,x)\).

\rev{
\begin{remark}[Threshold-metric formulation of the Lagrangian index]
\label{rem:lagindex-metrics}
Under thresholdability, the Lagrangian index in Remark~\ref{rem:relax-index-heuristic} can be written as
\[
d^*(x)=
\mathcal L\bigl(x,\langle1,z^*(\lambda^*)\rangle;\lambda^*\bigr)-
\mathcal L\bigl(x,\langle0,z^*(\lambda^*)\rangle;\lambda^*\bigr)
=f\bigl(x,z^*(\lambda^*)\bigr)-\lambda^*g\bigl(x,z^*(\lambda^*)\bigr).
\]
Thus, when \(g>0\), \(d^*(x)=g(x,z^*(\lambda^*))\{m(x,z^*(\lambda^*))-\lambda^*\}\). This representation is used below to characterize the sign and shape of \(d^*\), and computationally it means that both Lagrangian index policies are evaluated from the PCL objects \(f\), \(g\), and \(z^*(\lambda^*)\), without solving another belief-state dynamic program.
\end{remark}
}

The PCL-indexability conditions in \citet{nmmor20} are:
\begin{enumerate}[
  label=(PCLI\arabic*),
  widest=3,
  leftmargin=*,
  align=left,
  labelsep=0.6em,
  itemsep=2pt,
  topsep=2pt
]
\item The marginal work metric satisfies \(g(x,z)>0\) for all \(x\in[0,1]\) and \(z\in\overline{\mathbb R}\).
\item The MP index \(m(x)\) is nondecreasing and continuous in \(x\).
\item For each \(x\) and finite \(z_1<z_2\), the reward and work metrics are related by
\[
F(x,z_2)-F(x,z_1)=\int_{(z_1,z_2]} m(z)\,\mathrm{d}_z G(x,z).
\]
\end{enumerate}

\rev{ Here \(\mathrm{d}_z G(x,z)\) denotes the \emph{Lebesgue--Stieltjes} (LS) differential generated by
\(z\mapsto G(x,z)\), with \(x\) fixed; in this model the above LS integral (see \cite{cartvanBrunt00}) reduces to a sum
over threshold breakpoints.
The intuition behind \textup{(PCLI3)} is as follows. For fixed initial belief \(x\),
increasing the threshold \(z\) makes the threshold policy passive more often.
Therefore the threshold reward and work metrics can change only when the threshold
crosses a belief level that can affect the subsequent threshold trajectory. In the
present deterministic model, the possible jump points are contained in the countable
reachable set \(\mathcal D(x)\) of Proposition~\ref{pro:reachable-set}. Consequently,
as shown below in Proposition~\ref{pro:edrwm} and illustrated in
Figure~\ref{fig:FG3x}, \(F(x,\cdot)\) and \(G(x,\cdot)\) are right-continuous
step functions of the threshold.

For \(z\in\mathcal D(x)\), write the corresponding \emph{jumps} as
\[
\Delta_z F(x)\triangleq F(x,z)-F(x,z^-),\qquad
\Delta_z G(x)\triangleq G(x,z)-G(x,z^-),
\]
where \(F(x,z^-)\) and \(G(x,z^-)\) denote the corresponding left limits.
Thus, in this model, the LS identity \textup{(PCLI3)} reduces to a countable
sum over the threshold breakpoints:
\[
F(x,z_2)-F(x,z_1)
=
\sum_{z\in\mathcal D(x)\cap(z_1,z_2]}
m(z)\,\Delta_z G(x),
\]
where terms with zero jump contribute nothing. Hence, at every breakpoint $z$ with $\Delta_z G(x) < 0$,
\textup{(PCLI3)} says that \(\Delta_z F(x)=m(z)\Delta_z G(x)\), or, equivalently, 
\[
m(z)=
\frac{\Delta_z F(x)}{\Delta_z G(x)}
=
\frac{F(x,z^-)-F(x,z)}{G(x,z^-)-G(x,z)}.
\]
The last expression gives the MP index at the belief level $z$ as the discounted
reward lost per unit of discounted work saved when the threshold is raised through
that level. In this sense, \(m\) is the \emph{Radon--Nikodym derivative} of
the reward-induced LS measure with respect to the work-induced LS measure along the
threshold family; see \cite[Ch.~X]{doob94} and \cite{nmmor20}.
}

\begin{theorem}[\citet{nmmor20}]
\label{the:pcli}
Under \textup{(PCLI1--PCLI3)}, the model is threshold-indexable. Moreover, the Whittle index equals the MP index, \(w(x)=m(x)\), and the optimal threshold maps are the generalized inverses of \(m\).
\end{theorem}
Here the model is the single-patient \(\lambda\)-price problem. A \emph{generalized inverse} of the nondecreasing function \(m\) maps each price \(\lambda\in[m(0),m(1)]\) to the closure of \(m^{-1}(\{\lambda\})\); see \cite[p.~413]{falkTeschl12}. Once Sections~\ref{s:ewrm}--\ref{s:pcli3} verify \textup{(PCLI1--PCLI3)}, no value-function proof of threshold optimality or threshold monotonicity is needed.

\rev{
The Lagrangian benchmark policies introduced above rank states by \(d^*\),
whereas Whittle's policy ranks by the MP index \(m\). The next corollary
shows that, at the dual price, \(d^*\) has the same activation sign as the
MP index relative to the cutoff \(\lambda^*\). We use the following endpoint
convention: \(z^*(\lambda)\) is a generalized inverse of \(m\) for
\(0\le\lambda\leqslant m(1)\), extended by \(z^*(\lambda)=1\) for
\(\lambda>m(1)\), so prices above the largest index select the all-passive
threshold.

\begin{corollary}[Sign characterization of the Lagrangian index]
\label{cor:lagindex-sign}
Assume that \(\lambda^*\geqslant 0\) attains \eqref{eq:dualbound}. Under
\textup{(PCLI1--PCLI3)},
\[
\operatorname{sgn}(d^*(x))=\operatorname{sgn}(m(x)-\lambda^*).
\]
Thus \(d^*(x)\le0\) when \(m(x)\le\lambda^*\) and \(d^*(x)\geqslant 0\) when
\(m(x)\geqslant \lambda^*\); if \(\lambda^*\leqslant m(1)\), the sign is equivalently
determined by the threshold \(z^*(\lambda^*)\).
\end{corollary}

\begin{proof}
For \(0\le\lambda^*\leqslant m(1)\), Theorem~\ref{the:pcli} and
\cite[Lemma~24]{nmmor20} give
\(
\operatorname{sgn}(m(x,z^*(\lambda^*))-\lambda^*)=
\operatorname{sgn}(m(x)-\lambda^*)
\).
The result follows from Remark~\ref{rem:lagindex-metrics} and
\textup{(PCLI1)}. If \(\lambda^*>m(1)\), then \(z^*(\lambda^*)=1\),
Proposition~\ref{pro:dmrwm}\textup{(d)} gives
\(d^*(x)=rx/(1-\beta\rho)-\lambda^*<0\), and \(m(x)-\lambda^*<0\) for all \(x\).
\end{proof}

\begin{remark}[Policy implication]
Corollary~\ref{cor:lagindex-sign} explains why the standard Lagrangian index rule naturally leaves capacity unused when all \(d^*(x)\)'s are negative, whereas the forced-capacity version keeps the \(d^*\)-ranking but always selects the largest \(M\) values. The corollary only characterizes the sign of \(d^*\); its continuity, monotonicity, and piecewise-affine form are established later in Proposition~\ref{pro:lagindex-shape}, after Proposition~\ref{pro:zstar} gives the explicit threshold map.
\end{remark}
}

\section{Verification of the PCL-indexability conditions}
\label{s:avpclic}
We next verify conditions \textup{(PCLI1--PCLI3)} for the single-patient model. Fix $r>0$, $\beta, p \in(0,1)$, 
$q\in(0,1-p)$, 
$
\rho\triangleq 1-p-q\in(0,1)$, and
$h(x)\triangleq p+\rho x.
$

\subsection{Discounted reward and work metrics for threshold policies}
\label{s:ewrm}
We start by evaluating the reward and work metrics $F(x,z)$ and $G(x,z)$ under
the $z$-threshold policy, which are characterized by the functional equations
\begin{subequations}\label{eq:FGxz}
\begin{align}
F(x,z) &=
\begin{cases}
r+\beta\,F(p,z), & x>z,\\
r(1-x)+\beta\,F\!\bigl(h(x),z\bigr), & x\leqslant z,
\end{cases}
\label{eq:Fxz}\\[2mm]
G(x,z) &=
\begin{cases}
1+\beta\,G(p,z), & x>z,\\
\beta\,G\!\bigl(h(x),z\bigr), & x\leqslant z.
\end{cases}
\label{eq:Gxz}
\end{align}
\end{subequations}

To formulate their solution we define next certain quantities, including the \emph{passive trajectory} starting from $x$, $h_t(x) \triangleq z_\infty+(x-z_\infty)\rho^t$, with
\[
z_\infty\triangleq \frac{p}{1-\rho}=\frac{p}{p+q},
\quad
\textup{which solves} \quad
h_{t+1}(x)=h(h_t(x)),\ \ h_0(x)=x,
\]
and, for $z<z_\infty$, the corresponding \emph{first threshold-upcrossing time} from $x$, 
\[
\tau(x,z)\triangleq \min\{t\geqslant 0\colon h_t(x)>z\}.
\]

Noting that, for $x\leqslant z$, $h_t(x)>z \Longleftrightarrow \rho^{t}<\dfrac{z_\infty-z}{z_\infty-x}$, yields the
formulas
\begin{equation}
\label{eq:tauxz}
\tau(x,z)
\;=\;
\left\lfloor
\frac{\ln\!\bigl(\frac{z_\infty-z}{z_\infty-x}\bigr)}{\ln \rho}
\right\rfloor+1,
\quad x\leqslant z<z_\infty,
\qquad
\tau(x,z)=0,\quad x>z.
\end{equation}
Further, since $h_t(x)$ is an increasing sequence converging to $z_\infty$ as $t \to \infty$,
\begin{equation}
\label{eq:tau-characterization}
\tau(x,z)=t
\iff
h_{t-1}(x)\leqslant z < h_t(x), \quad x\leqslant z <z_\infty.
\end{equation}

We will further need the threshold breakpoints
\[
z_{-1}=0,\quad z_t\triangleq h_t(p)=h_{t+1}(0),\ t=0,1,\ldots,
\quad z_t\nearrow z_\infty,
\]
so that, for $p\leqslant z<z_\infty$, $t=\tau(p, z)$ is characterized by
$z_{t-1}\leqslant z<z_t$.

Finally, define the discounted passive sums
\begin{equation}
\label{eq:Phi-t-p-sum}
\Phi_t(x)\triangleq \sum_{s=0}^{t-1}(1-h_s(x))\beta^s, \quad
\Phi_\infty(x)\triangleq \sum_{s=0}^{\infty}(1-h_s(x))\beta^s.
\end{equation}
Note that $F(x, 1) = r \Phi_\infty(x)$ is the reward metric under the always-passive policy. 
They admit the
closed forms
\[
\Phi_t(x)=
\frac{1-\beta^t}{1-\beta}(1-z_\infty)
+\frac{1-(\beta\rho)^t}{1-\beta\rho}(z_\infty-x),
\quad
\Phi_\infty(x)=
\frac{1-z_\infty}{1-\beta}+\frac{z_\infty-x}{1-\beta\rho}.
\]

The next result gives closed-form expressions for $F(x,z)$ and $G(x,z)$ and highlights their staircase
structure in $z$ (piecewise constant for fixed $x$).

\begin{proposition}\label{pro:edrwm}
$F(x, z)$ and $G(x, z)$ are given by:
\begin{enumerate}[label=(\alph*)]
\item If $z<p$, then
\[
F(x,z)=
\begin{cases}
r\bigl(\frac{1}{1-\beta}-x\bigr), & x\leqslant z,\\
\frac{r}{1-\beta}, & x>z,
\end{cases}
\quad
G(x,z)=
\begin{cases}
\frac{\beta}{1-\beta}, & x\leqslant z,\\
\frac{1}{1-\beta}, & x>z.
\end{cases}
\]
\item If $p\leqslant z<z_\infty$ and $t=\tau(p, z)$, i.e., 
$z_{t-1}\leqslant z<z_t$, then
\[
F(x,z)=
\begin{cases}
r\Phi_s(x)+\beta^s\bigl[r+\beta F(p,z)\bigr],
& x\leqslant z,\ h_{s-1}(x)\leqslant z<h_s(x),\\
r+\beta F(p,z), & x>z,
\end{cases}
\]
\[
G(x,z)=
\begin{cases}
\beta^s\bigl[1+\beta G(p,z)\bigr],
& x\leqslant z,\ h_{s-1}(x)\leqslant z<h_s(x),\\
1+\beta G(p,z), & x>z,
\end{cases}
\]
where
$
F(p,z)=r\,\frac{\Phi_t(p)+\beta^t}{1-\beta^{t+1}},
\quad
G(p,z)=\frac{\beta^t}{1-\beta^{t+1}}.
$
\item If $z_\infty\leqslant z<1$, then
\[
F(x,z)=
\begin{cases}
r\Phi_\infty(x), & x\leqslant z,\\
r\bigl[1+\beta\Phi_\infty(p)\bigr], & x>z,
\end{cases}
\quad
G(x,z)=
\begin{cases}
0, & x\leqslant z,\\
1, & x>z.
\end{cases}
\]
\item If $z\geqslant 1$, then $F(x,z)=r\Phi_\infty(x)$ and $G(x,z)=0$ for all $x$.
\end{enumerate}
\end{proposition}

The proof of Proposition~\ref{pro:edrwm} is in Appendix~A.1.
Figure~\ref{fig:FG3x} illustrates the staircase dependence on the threshold $z$ for representative $x$ values in the regimes $x<p$, $p<x<z_\infty$, and $x>z_\infty$, for $p=0.3$, $q=0.2$, $r=1$, $\beta=0.95$ (used again later).
Both $F(x,z)$ and $G(x,z)$ are right-continuous, piecewise constant in $z$, with infinitely many steps accumulating at $z_\infty$.

\begin{figure}[!htbp]
  \centering
  \includegraphics[width=1\textwidth]{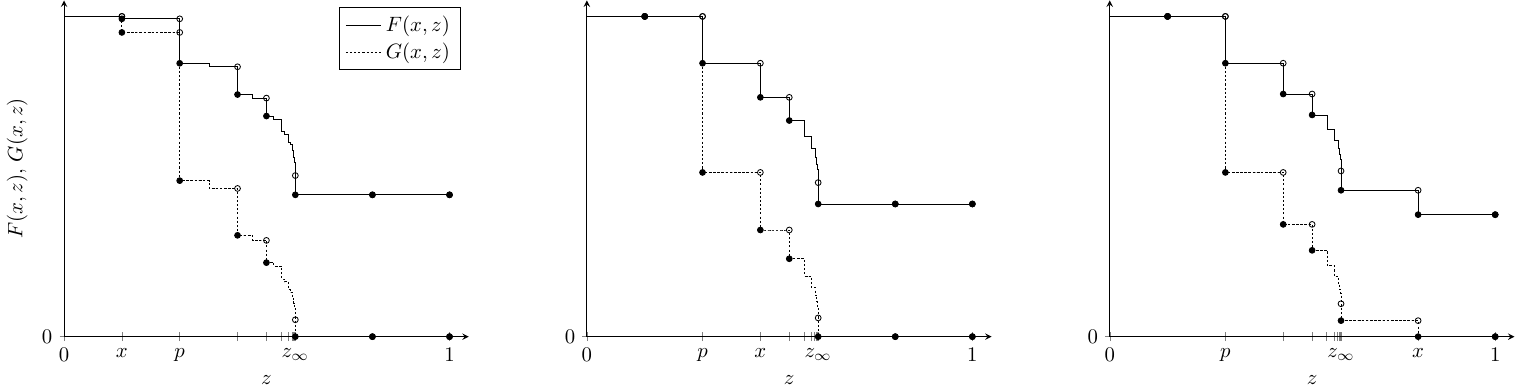}%
  \caption{Reward and work metrics
    as functions of the threshold $z$ for fixed belief $x$.}
  \label{fig:FG3x}
\end{figure}

\subsection{Discounted marginal metrics and condition \textup{(PCLI1)}}
\label{s:edmrwm}
We proceed by evaluating the marginal reward and work metrics $f(x,z)$ and $g(x,z)$ defined in Section~\ref{s:pclbvtri};
as with $F$ and $G$, for fixed $x$ these marginal metrics exhibit a staircase dependence on the
threshold $z$.

Throughout this subsection, we use the shorthand
\[
K_F(z)\triangleq r+\beta F(p,z),
\qquad
K_G(z)\triangleq 1+\beta G(p,z),
\]
where \(F(p,z)\) and \(G(p,z)\) are the  reward and work metrics
given in Proposition~\ref{pro:edrwm}.

\begin{proposition}\label{pro:dmrwm}
$f(x, z)$ and $g(x, z)$ are given by:

\begin{enumerate}[label=(\alph*)]
\item If $z<p$, then for all $x$,
$
f(x,z)=r x,
\quad
g(x,z)=1.
$

\item If $p\leqslant z<z_\infty$, let $t=\tau(p,z)$, i.e., $z_{t-1}\leqslant z<z_t$. Then:
\begin{itemize}[leftmargin=*]
\item If $x\leqslant z$ and $s=\tau(x,z)$, 
\[
f(x,z)=\bigl(1-\beta^s\bigr)K_F(z)-r\,\Phi_s(x),
\quad
g(x,z)=\bigl(1-\beta^s\bigr)K_G(z).
\]
\item If $x>z$, 
\[
f(x,z)= (1-\beta)K_F(z) - r(1-x),
\quad
g(x,z)=\bigl(1-\beta\bigr)K_G(z).
\]
\end{itemize}

\item If $z_\infty\leqslant z<1$, then $F(p,z)=r\,\Phi_\infty(p)$ and $G(p,z)=0$.
Define
\[
x^*(z)\triangleq h^{-1}(z)=\frac{z-p}{\rho}.
\]
\begin{itemize}[leftmargin=*]
\item If $x>z$ and $x> x^*(z)$ (equivalently $h(x)> z$), then
\[
f(x,z)=r(x-1)+(1-\beta)\,K_F(z),
\quad
g(x,z)=1-\beta.
\]
\item Otherwise, $f(x,z)=\dfrac{r\,x}{1-\beta\rho}$ \quad and \quad $g(x,z)=1$.
\end{itemize}

\item If $z\geqslant 1$, then 
$f(x,z)=\dfrac{r\,x}{1-\beta\rho}$ \quad and \quad $g(x,z)=1$.
\end{enumerate}
\end{proposition}

The proof of Proposition~\ref{pro:dmrwm} is in Appendix~A.2.
Figure~\ref{fig:margfg3x} plots the marginal reward and work metrics versus $z$ for fixed $x$, in the regimes
$x<p$, $p<x<z_\infty$, and $x>z_\infty$.
Both $f(x,z)$ and $g(x,z)$ are right-continuous, piecewise-constant in $z$, with infinitely many steps accumulating at $z_\infty$.

\begin{figure}[!htbp]
  \centering
  \includegraphics[width=1\textwidth]{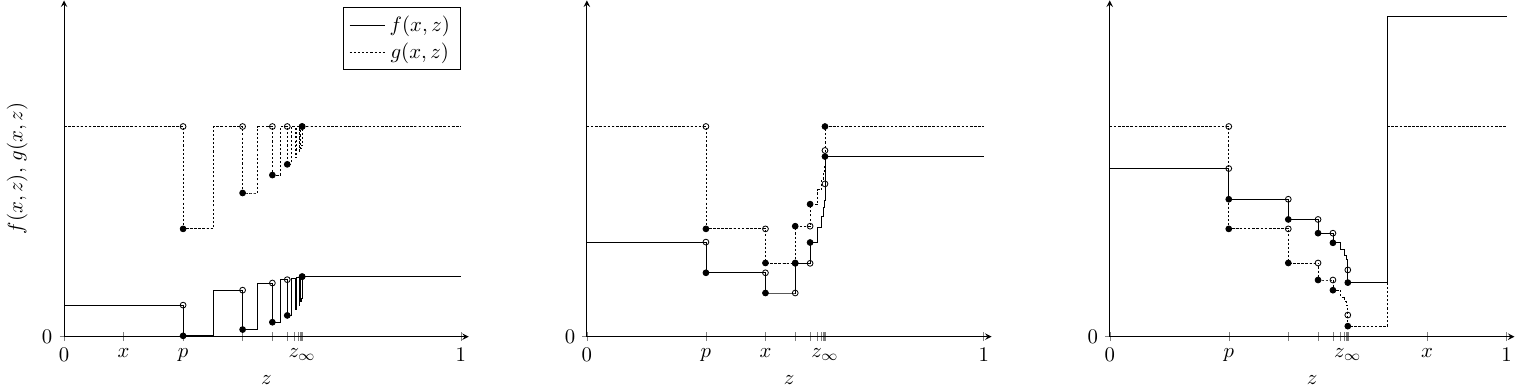}%
  \caption{Marginal reward and work metrics
    as functions of the threshold $z$ for fixed $x$.}
  \label{fig:margfg3x}
\end{figure}

The figure suggests that $g(x,z)$ is bounded away from zero, consistently with condition \textup{(PCLI1)}.
The next result ensures that \textup{(PCLI1)} indeed holds.

\begin{lemma}[PCLI1]
\label{lma:pcli1}
For all $x$ and $z$, $g(x,z)\geqslant 1-\beta$. Hence \textup{(PCLI1)} holds.
\end{lemma}

\begin{proof}
By Proposition~\ref{pro:dmrwm}, if $z<p$ or $z\geqslant1$, then $g(x,z)=1$. If $z_\infty\leqslant z<1$, then $g(x,z)$ is either $1-\beta$ or $1$. If $p\leqslant z<z_\infty$ and $x>z$, then $g(x,z)=(1-\beta)K_G(z)\ge1-\beta$. It remains to consider $p\leqslant z<z_\infty$ and $x\leqslant z$, where
$g(x,z)=(1-\beta^{s})(1+\beta G(p,z))$ for $s=\tau(x,z)\geqslant 1$.
Since $s\geqslant 1$ we have $\beta^{s}\leqslant \beta$ and thus $1-\beta^{s}\geqslant 1-\beta$.
Moreover, $G(p,z)\geqslant 0$, so $1+\beta G(p,z)\geqslant 1$. Therefore,
\[
g(x,z)=(1-\beta^{s})(1+\beta G(p,z))\;\geqslant\; 1-\beta.
\]
\end{proof}

\subsection{MP index and condition \textup{(PCLI2)}}
\label{s:pcli2}

Since (PCLI1) holds by Lemma~\ref{lma:pcli1}, we can define the MP metric $m(x,z)\triangleq f(x,z)/g(x,z)$ and the MP index
$m(x)\triangleq m(x,x)$. To evaluate the latter,
it suffices to give $m(x,z)$
for $0\leqslant x\leqslant z\leqslant 1$.

\begin{proposition}\label{pro:dmpm}
For $0\leqslant x\leqslant z\leqslant 1$, the MP metric is given by:
\begin{enumerate}[label=(\alph*)]
\item If $z<p$, then $m(x,z)=r x$.
\item If $p\leqslant z<z_\infty$ and $s=\tau(x,z)$, then, with $F(p,z),G(p,z)$ as in Proposition~\ref{pro:edrwm}\textup{(b)}, 
\[
m(x,z)=
\frac{(1-\beta^{s})(r+\beta F(p,z))-r\Phi_s(x)}
     {(1-\beta^{s})(1+\beta G(p,z))}.
\]
\item If $z_\infty\leqslant z\leqslant 1$, then
$m(x,z)=r\bigl[1+\beta\Phi_\infty(p)-\Phi_\infty(x)\bigr]$.
\end{enumerate}
\end{proposition}
\begin{proof}[Proof sketch]
Immediate from Proposition~\ref{pro:dmrwm} using $m(x,z)\triangleq f(x,z)/g(x,z)$
and the corresponding expressions for $f(x,z)$ and $g(x,z)$ in each regime.
\end{proof}

The next result specializes the MP metric to the diagonal $z=x$, yielding an explicit piecewise-affine
formula for the MP index.

\begin{proposition}\label{pro:dmpi-explicit}
The MP index $m(x)$ is the piecewise-affine function
\[
m(x)=
\begin{cases}
r x, & 0\leqslant x<p,\\[4pt]
\begin{aligned}
\displaystyle \frac{r}{1-\beta}\bigl[(1-\beta^{t+1})x
-\beta(1-\beta^{t}) +\beta(1-\beta)\Phi_t(p)\bigr],
\end{aligned}
& z_{t-1}\leqslant x<z_t,\ t \geqslant 1,\\[8pt]
\displaystyle \frac{r x}{1-\beta\rho}, & z_\infty\leqslant x\leqslant 1.
\end{cases}
\]
\end{proposition}

\begin{proof}[Proof sketch]
The formula is obtained by substituting \(z=x\) in Proposition~\ref{pro:dmpm}. The cases \(x<p\) and \(x\geqslant z_\infty\) give \(rx\) and \(rx/(1-\beta\rho)\) directly. For \(p\leqslant x<z_\infty\), let \(t=\tau(p,x)\), equivalently \(z_{t-1}\leqslant x<z_t\). Since \(\tau(x,x)=1\), Proposition~\ref{pro:dmpm}\textup{(b)} together with Proposition~\ref{pro:edrwm}\textup{(b)} yields the displayed affine expression after algebra.
\end{proof}

We next verify \textup{(PCLI2)}; in fact the MP index is strictly increasing.

\begin{lemma}[PCLI2]\label{lma:pcli2}
The MP index $m(x)$ is continuous and increasing in $x$, and hence
condition \textup{(PCLI2)} holds.
\end{lemma}

\begin{proof}[Proof sketch]
Proposition~\ref{pro:dmpi-explicit} shows that $m$ is affine on each interval
$(0,p)$, $(z_{t-1},z_t)$, and $(z_\infty,1)$ with positive slope, and adjacent
branches match at the breakpoints by construction. Continuity and strict
monotonicity follow. Details are given in Appendix A.3.
\end{proof}

Figure~\ref{fig:mpi} illustrates the resulting piecewise-affine index.

\begin{figure}[!htbp]
  \centering
  \includegraphics[width=0.45\textwidth]{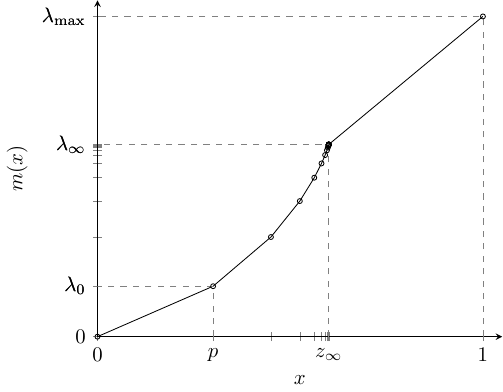}%
  \caption{MP index $m(x)$ versus belief state $x$.}
  \label{fig:mpi}
\end{figure}

\subsection{Verifying \textup{(PCLI3)} and threshold-indexability}
\label{s:pcli3}

To verify the remaining condition \textup{(PCLI3)}, we invoke a sufficient criterion from
\cite[Proposition~6]{nmmor20}: if \textup{(PCLI1--PCLI2)} hold and, for each initial belief state $x$,
every threshold-policy trajectory is confined to a countable closed set $\mathcal{D}(x)$ that forms
the endpoint set of a partition of $[0,1]$ into left-semiclosed intervals and singletons, then
\textup{(PCLI3)} holds.

\begin{proposition}\label{pro:reachable-set}
For each $x\in[0,1]$, every belief trajectory under any threshold policy is contained in the
countable closed set
\[
\mathcal{D}(x)=\{0,1\}\cup\{h_t(x):t\geqslant 0\}\cup\{z_t:t\geqslant 0\}\cup\{z_\infty\}.
\]
\end{proposition}

\begin{proof}[Proof sketch]
Under passivity the belief evolves along the trajectory $\{h_t(x)\}_{t\geqslant 0}$.
Under activation the belief resets to $p$, and subsequent passive evolution visits the points
$z_t=h_t(p)$. Hence every visited belief lies in $\{h_t(x)\}_{t\geqslant 0}\cup\{z_t\}_{t\geqslant 0}$, together with
the endpoints $\{0,1\}$. The only accumulation point is $z_\infty$, so
$\mathcal{D}(x)$ is closed and countable. Details are in Appendix~A.4.
\end{proof}

\begin{lemma}[PCLI3]
\label{lma:pcli3}
The adherence model satisfies \textup{(PCLI3)}.
\end{lemma}

\begin{proof}
By Lemmas~\ref{lma:pcli1} and~\ref{lma:pcli2}, \textup{(PCLI1--PCLI2)} hold.
Proposition~\ref{pro:reachable-set} gives the required set
$\mathcal{D}(x)$ for each $x$. Hence \textup{(PCLI3)} follows from \cite[Proposition~6]{nmmor20}.
\end{proof}

We now combine the PCL conditions with the MP index formulas to establish
threshold-indexability of the model and identify its Whittle index.

\begin{theorem}\label{the:adherence-indexability}
The discounted adherence restless bandit model is threshold-indexable. Its Whittle index equals its MP index:
$w(x)=m(x)$ for all $x\in[0,1]$.
\end{theorem}

\begin{proof}
The model satisfies \textup{(PCLI1--PCLI3)}, so Theorem~\ref{the:pcli} yields the result.
\end{proof}

\subsection{The optimal threshold map $z^*(\lambda)$}
\label{s:otfl}
To compute the Lagrangian relaxation and dual bound, we study next the single-patient $\lambda$-price subproblem in which an
intervention incurs price $\lambda$. By Theorem~\ref{the:adherence-indexability}, an optimal policy is a belief-threshold
rule, so it suffices to characterize the optimal threshold map $z^*(\lambda)$. The next proposition gives $z^*(\lambda)$
explicitly as the inverse of the MP index $m(\cdot)$ (see Theorem~\ref{the:pcli}).

Recall that the MP index $m(\cdot)$ is continuous and increasing by
Lemma~\textup{\ref{lma:pcli2}}. Hence $m(\cdot)$ is a bijection from $[0,1]$ onto
$[0,m(1)]$, and the \emph{optimal threshold map} is unique and given by its inverse
$
z^*(\lambda)\;\triangleq\;m^{-1}(\lambda)$, for $\lambda\in[0,m(1)],$
which is continuous and increasing.
To characterize $z^*(\lambda)$, we write
\[
C_t
\;\triangleq\;
\beta^{t+1}-\beta+\beta(1-\beta)\,\Phi_t(p),
\]
and consider 
the state-space breakpoints
\[
0,\quad z_0 = p,\quad z_1,\quad z_2,\ \ldots,\ z_t,\ \ldots,\ z_\infty,\quad 1.
\]
Define the corresponding index values at the threshold breakpoints by
\[
\lambda_t \;\triangleq\; m(z_t)\quad (t=0, 1,2,\ldots),\quad
\lambda_\infty \;\triangleq\; m(z_\infty),\quad
\lambda_{\max} \;\triangleq\; m(1)=\frac{r}{1-\beta\rho}.
\]
These $\lambda_t$ form the breakpoints in price space. For later aggregate calculations with heterogeneous patients, we use the endpoint convention
\[
z^*(\lambda)=1,\qquad \lambda\geqslant \lambda_{\max},
\]
so that prices above a patient's largest index select the all-passive threshold for that patient. On \([0,\lambda_{\max}]\), the inverse is given explicitly as follows.

\begin{proposition}[Optimal threshold map]
\label{pro:zstar}
For $\lambda\in[0,\lambda_{\max}]$, 
\[
z^*(\lambda)=
\begin{cases}
\dfrac{\lambda}{r}, & 0\leqslant \lambda<\lambda_0,\\[10pt]
\displaystyle
\dfrac{1-\beta}{r\bigl(1-\beta^{t+1}\bigr)}\,\lambda
-\dfrac{C_t}{1-\beta^{t+1}},
& \lambda_{t-1}\leqslant \lambda<\lambda_t,\ \ t=1,2,\ldots,\\[10pt]
\dfrac{1-\beta\rho}{r}\,\lambda, & \lambda_\infty\leqslant \lambda\leqslant \lambda_{\max}.
\end{cases}
\]
\end{proposition}
For the dual computation we extend the selector by setting $z^*(\lambda)=1$ for $\lambda\geqslant \lambda_{\max}$, the all-passive threshold on $[0,1]$. Thus $z_n^*(\lambda)$ is defined for every price $\lambda \geqslant 0.$
\begin{proof}[Proof outline]
By Lemma~\ref{lma:pcli2}, \(m\) is continuous and increasing from \([0,1]\) onto \([0,\lambda_{\max}]\), hence has inverse \(z^*\). Proposition~\ref{pro:dmpi-explicit} gives affine branches on \([0,p)\), \([z_{t-1},z_t)\), and \([z_\infty,1]\); solving \(\lambda=m(x)\) on each branch gives the displayed formulas. Continuity of \(z^*\) follows from that of \(m\).
\end{proof}

Figure~\ref{fig:zstar} illustrates the piecewise-affine threshold map used below in the Lagrangian bound computation.

\begin{figure}[!htbp]
  \centering
  \includegraphics[width=0.45\textwidth]{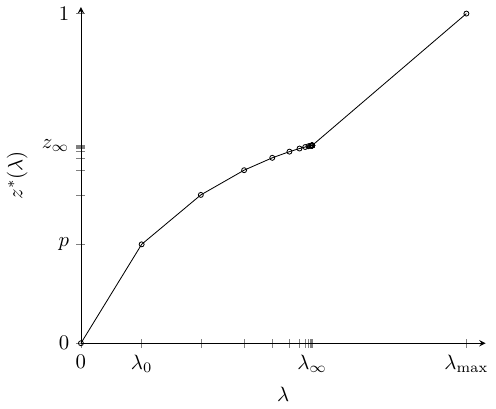}%
  \caption{Optimal threshold $z^*(\lambda)$ versus intervention price $\lambda$.}
  \label{fig:zstar}
\end{figure}

\rev{Returning to the alternative Lagrangian benchmark introduced in
Remark~\ref{rem:relax-index-heuristic}, and to its sign characterization in
Corollary~\ref{cor:lagindex-sign}, the formulas for the marginal
metrics yield a simple structural description of the Lagrangian index as a
function of the belief state.}

\rev{
\begin{proposition}\label{pro:lagindex-shape}
For every fixed \(\lambda\geqslant 0\), letting \(z^*(\lambda)=1\) for \(\lambda\geqslant \lambda_{\max}\), define
\[
d^\lambda(x)\;\triangleq\;f\bigl(x,z^*(\lambda)\bigr)
-\lambda\,g\bigl(x,z^*(\lambda)\bigr),
\qquad x\in[0,1].
\]
Then \(d^\lambda(\cdot)\) is continuous, increasing, and piecewise affine on
\([0,1]\). In particular, the Lagrangian index
\(d^*(x)=d^{\lambda^*}(x)\) is continuous, increasing, and piecewise affine.
\end{proposition}

\begin{proof}
Fix \(\lambda\geqslant 0\) and put \(z=z^*(\lambda)\). If \(\lambda<\lambda_{\max}\), then \(\lambda=m(z)\); if \(\lambda\geqslant \lambda_{\max}\), then \(z=1\). Substitution into Proposition~\ref{pro:dmrwm} gives four regimes. For \(z<p\) and \(z\ge1\), \(d^\lambda(x)=rx-\lambda\) and \(d^\lambda(x)=rx/(1-\beta\rho)-\lambda\). For \(p\leqslant z<z_\infty\), with \(A_\lambda(z)=K_F(z)-\lambda K_G(z)\),
\[
d^\lambda(x)=
\begin{cases}
(1-\beta^s)A_\lambda(z)-r\Phi_s(x), & x\leqslant z,\ s=\tau(x,z),\\
(1-\beta)A_\lambda(z)-r(1-x), & x>z.
\end{cases}
\]
Thus \(d^\lambda\) is affine on the intervals determined by \(\tau(x,z)\); the slopes are \(r\{1-(\beta\rho)^s\}/(1-\beta\rho)>0\) on the passive-spell pieces and \(r>0\) on \((z,1]\). Adjacent pieces match because \(\Phi_{s+1}(h^{-s}(z))=\Phi_s(h^{-s}(z))+\beta^s(1-z)\) and \(\lambda=m(z)\) imply \((1-\beta)A_\lambda(z)=r(1-z)\). For \(z_\infty\leqslant z<1\), Proposition~\ref{pro:dmrwm}\textup{(c)} gives two affine pieces with slopes \(r/(1-\beta\rho)\) and \(r\); they match at \(h^{-1}(z)\) by \(\lambda=m(z)=rz/(1-\beta\rho)\). Hence \(d^\lambda\) is continuous, increasing, and piecewise affine. Taking \(\lambda=\lambda^*\) gives the claim for \(d^*\).
\end{proof}
}

Figure~\ref{fig:lagindex} illustrates the resulting continuous, increasing, piecewise-affine Lagrangian index.

\begin{figure}[!htbp]
  \centering
  \includegraphics[width=0.45\textwidth]{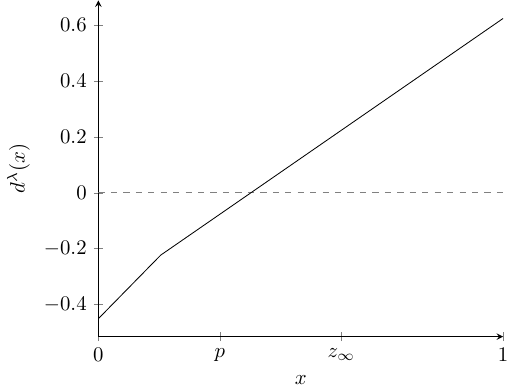}%
  \caption{Lagrangian index \(d^\lambda(x)\) versus belief \(x\) for
  \(\lambda=0.45\).}
  \label{fig:lagindex}
\end{figure}

\section{Analytic Lagrangian relaxation and dual bound computation}
\label{s:dual-bound}
This section uses the PCL analysis to obtain an analytic Lagrangian relaxation and computable dual bound. The explicit threshold map and threshold metrics give the single-patient \(\lambda\)-price value in closed form for any multiplier; aggregating these values and minimizing by bisection gives the multi-patient bound. \rev{ The same marginal metrics yield the Lagrangian indices \(d_n^*(\cdot)\) used by the standard and forced-capacity benchmarks.} We also specialize the bound to a uniform initial belief distribution, producing the population-averaged benchmark used in the experiments.

\subsection{Population-averaged metrics under a uniform initial belief}
\label{s:uniform}
We begin by removing the dependence on a specific initial belief through a population-averaging step. Concretely, we assume
the initial belief is drawn from $\nu_0=\mathrm{Unif}[0,1]$ and compute the corresponding averaged discounted reward and work
metrics under a threshold policy. These closed-form expressions will serve as basic building blocks for the uniform-initial
Lagrangian relaxation and the dual bound used in our computational experiments.

We write the corresponding population-averaged discounted reward and work metrics for patient $n$ under the
$z$-threshold policy as
\[
F_n(z)\triangleq \int_0^1 F_n(x,z)\, \mathrm{d}x,
\quad
G_n(z)\triangleq \int_0^1 G_n(x,z)\,\mathrm{d}x,
\]
which admit closed-form expressions (piecewise in $z$). These formulas are used in
Sections~\ref{s:single-lagrangian},~\ref{s:bisection}, and~\ref{s:numerics} to compute Lagrangian bounds and generate the
numerical results, and are given in the Online Supplement, Corollary~1. 

\subsection{Single-patient piecewise-affine Lagrangian value}
\label{s:single-lagrangian}
We next express the relaxed objective as an explicit convex function of the per-intervention price $\lambda$.
The key observation is that, once the single-patient $\lambda$-price problem is solved by a threshold
$z_n^*(\lambda)$ (Theorem~\ref{the:adherence-indexability} and Proposition~\ref{pro:zstar}), the single-patient Lagrangian
optimal value is the pointwise supremum of affine functions of $\lambda$, and its right derivative equals the slope of the
active affine piece (cf.\ \cite[Lemma~4]{nmmor20}).

For a patient $n$ and threshold $z$, define the Lagrangian performance
\[
\mathcal{L}_n(z;\lambda)\triangleq F_n(z)-\lambda\,G_n(z).
\]
The single-patient Lagrangian optimal value is
\[
L_n(\lambda)\triangleq \sup_{z\in[0,1]}\mathcal{L}_n(z;\lambda)
= F_n\!\bigl(z_n^*(\lambda)\bigr)-\lambda\,G_n\!\bigl(z_n^*(\lambda)\bigr),
\]
where $z_n^*(\lambda)$ is the optimal threshold selector (Proposition~\ref{pro:zstar}).
Thus $L_n(\cdot)$ is convex and piecewise affine, and its right derivative $L'_{n,+}(\lambda)$ satisfies
\begin{equation}\label{eq:ell_n-lambda}
L'_{n,+}(\lambda)\triangleq \frac{d^+}{d\lambda}L_n(\lambda)
=-\,G_n\!\bigl(z_n^*(\lambda)\bigr).
\end{equation}

Further, the aggregate Lagrangian dual function is
\[
L(\lambda)=\sum_{n=1}^{N}L_n(\lambda)+\frac{M}{1-\beta}\,\lambda,
\]
and hence its right derivative is given by
\begin{equation}\label{eq:ell-lambda}
L'_+(\lambda)\triangleq \frac{d^+}{d\lambda}L(\lambda)
=-\sum_{n=1}^{N}G_n\!\bigl(z_n^*(\lambda)\bigr)+\frac{M}{1-\beta}.
\end{equation}
Since $L(\cdot)$ is convex and piecewise affine, $L'_+(\lambda)$ exists for all $\lambda$ and is a nondecreasing,
piecewise-constant function of $\lambda$.

For the aggregate search, set
\[
\lambda_{\max}\triangleq\max_{1\leqslant n\leqslant N}m_n(1)
=\max_{1\leqslant n\leqslant N}\frac{r_n}{1-\beta\rho_n}.
\]
For prices above patient \(n\)'s value of \(m_n(1)\), the endpoint convention following Proposition~\ref{pro:zstar} gives \(z_n^*(\lambda)=1\) and \(G_n(z_n^*(\lambda))=0\).

Equation~\eqref{eq:ell-lambda} makes computation straightforward: for given $\lambda$, we evaluate $z_n^*(\lambda)$ for
each patient $n$ (via Proposition~\ref{pro:zstar} and the endpoint convention) and then compute $G_n\!\bigl(z_n^*(\lambda)\bigr)$. Summing these terms
yields $L'_+(\lambda)$, with per-evaluation cost $O(N)$.

Note that any minimizer $\lambda^*$ of $L(\cdot)$ on $[0,\lambda_{\max}]$ satisfies $0\in\partial L(\lambda^*)$.
Equivalently:
(1) if $L'_+(0)\geqslant 0$, then $\lambda^*=0$;
(2) if $L'_+(\lambda_{\max})\leqslant 0$, $\lambda^*=\lambda_{\max}$; and
(3) otherwise, there exists $\lambda^*\in(0,\lambda_{\max})$ with
      $L'_+(\lambda^-)\leqslant 0\leqslant L'_+(\lambda^+)$.

We exploit the monotonicity of $L'_+(\lambda)$ to compute the optimal price $\lambda^*$ by a
one-dimensional bisection method, without enumerating breakpoints.

\subsection{Multi-patient dual bound computation via bisection}
\label{s:bisection}
For the uniform initial belief distribution, the population-averaged dual problem is
\[
\lambda^*\in\argmin_{\lambda\in[0,\lambda_{\max}]}L(\lambda),
\qquad
\lambda_{\max}=\max_n m_n(1)=\max_n\frac{r_n}{1-\beta\rho_n}.
\]
The same procedure applies to a fixed initial state \(\mathbf{x}^0\) by replacing \(L(\lambda)\) and \(G_n(z)\) with \(L(\mathbf{x}^0;\lambda)\) and \(G_n(x_n^0,z)\). Since \(L\) is convex and \(L'_+\) in \eqref{eq:ell-lambda} is nondecreasing, the minimizer is found as follows. If \(L'_+(0)\geqslant 0\), set \(\lambda^*=0\); if \(L'_+(\lambda_{\max})\le0\), set \(\lambda^*=\lambda_{\max}\). Otherwise, bisect the bracket \([0,\lambda_{\max}]\), evaluating \(L'_+(\lambda)\) by \eqref{eq:ell-lambda}, until either \(|L'_+(\lambda)|\le\varepsilon\) or the bracket length is at most \(\varepsilon\); then evaluate \(L(\lambda^*)\) from the single-patient values in Section~\ref{s:single-lagrangian}.

Each iteration evaluates \(z_n^*(\lambda)\) and \(G_n(z_n^*(\lambda))\) for all patients, so it costs \(O(N)\). The total cost is therefore \(O\!\bigl(N\log(1/\varepsilon)\bigr)\). Countably many breakpoints, with possible accumulation, cause no numerical difficulty because the algorithm uses only pointwise evaluations of the monotone derivative.

\rev{
After \(\lambda^*\) is computed, the Lagrangian index benchmarks require no further dynamic programming. For any patient state \(x\),
\[
 d_n^*(x)=f_n\bigl(x,z_n^*(\lambda^*)\bigr)-\lambda^* g_n\bigl(x,z_n^*(\lambda^*)\bigr).
\]
The standard Lagrangian policy activates positive \(d_n^*(x)\)'s up to capacity, while the forced-capacity policy activates the \(M\) largest values. Thus the same PCL formulas compute the Whittle index, the dual bound, and both Lagrangian index benchmarks.
}

\section{Dependence of the Whittle index on model parameters}
\label{s:dmpimp}
By Theorem~\ref{the:adherence-indexability}, the Whittle index equals the MP index, denoted \(m\). We summarize how \(m(x;p,q)\) varies with the lapse probability \(p\) and spontaneous-recovery probability \(q\), the two transition parameters governing passive adherence dynamics.

\subsection{Dependence on the lapse probability from adherence \(p\)}
\label{s:dilpp}
Fix \(q\) and \(x\). Although larger \(p\) makes adherence more fragile, it also changes the passive update \(h(x)=p+(1-p-q)x\). The explicit index formula nevertheless gives a global monotonicity result.

\begin{proposition}
\label{pro:mp-monotone-p}
\(m(x;p,q)\), as a function of \(p\in(0,1-q)\), is nonincreasing on \((0,1-q)\), decreasing on \((0,\min\{x,1-q\})\), and constant on \([x,1-q)\).
\end{proposition}
The proof is in Appendix~C.1. Figure~\ref{fig:mpvspx} illustrates the result.

\begin{figure}[!htbp]
  \centering
  \includegraphics[width=0.45\textwidth]{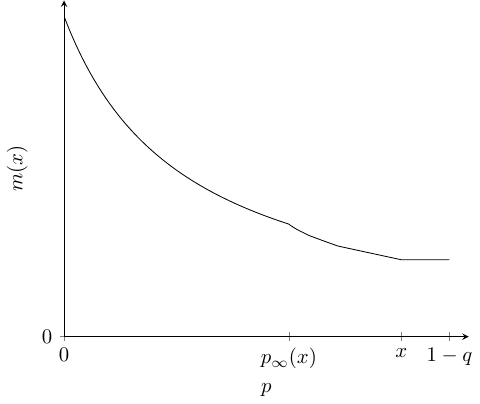}%
  \caption{Dependence of the Whittle index \(m(x)\) on the lapse probability \(p\) for fixed \(x\).}
  \label{fig:mpvspx}
\end{figure}

\begin{remark}[Interpretation of the \(p\)-dependence]
\label{re:mp-monotone-p}
For \(p\geqslant x\), the state lies on the first branch and \(m(x)=rx\); for \(p<x\), increasing \(p\) reduces the expected benefit of intervention and the index decreases.
\end{remark}

\subsection{Dependence on the spontaneous recovery probability \(q\)}
\label{s:dsrpq}
Fix \(p\). Unlike the dependence on \(p\), the dependence on \(q\) is not globally monotone; the sign changes across the branch regions induced by \(h_t(p,q)\) and \(z_\infty(p,q)=p/(p+q)\).

\begin{proposition}
\label{pro:mp-monotone-q}
Consider \(m(x;p,q)\) as a function of \(q\in(0,1-p)\).
\begin{enumerate}[label=(\roman*)]
\item If \(0\leqslant x<p\), then \(m(x;p,q)=rx\), independent of \(q\).
\item If \(x\geqslant p\), then on the first middle branch \(p\leqslant x<h_1(p,q)\) the index is independent of \(q\), while on every higher middle branch \(h_{t-1}(p,q)\leqslant x<h_t(p,q)\), \(t\ge2\), it is increasing in \(q\).
\item If \(x>p\) and \(q\geqslant q_\infty(x)\triangleq p(1-x)/x\), equivalently \(z_\infty(p,q)\leqslant x\), the index is decreasing in \(q\).
\end{enumerate}
Thus, for fixed \(p\) and \(x>p\), the index is locally flat on the first middle branch, increases on higher middle branches, and eventually decreases on the last branch.
\end{proposition}
The proof is in Appendix~C.2. Figure~\ref{fig:mpvsq2x} illustrates the two cases \(x\leqslant p\) and \(x>p\).

\begin{figure}[!htbp]
  \centering
  \includegraphics[width=0.8\textwidth]{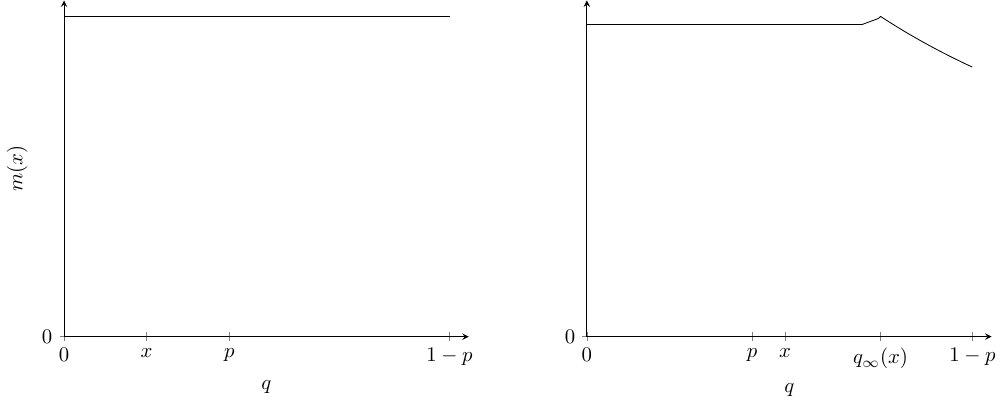}%
  \caption{Dependence of \(m(x)\) on the spontaneous recovery probability \(q\) for fixed \(x\).}
  \label{fig:mpvsq2x}
\end{figure}

The nonmonotonicity reflects two effects of passive recovery. In intermediate belief regions, higher \(q\) can raise marginal productivity by improving the post-passive trajectory; in high-belief regions, passivity becomes a closer substitute for intervention and the index falls.

\rev{
\section{Long-run average criterion: Whittle indexability}
\label{s:iatac}

This section extends the single-patient analysis from the discounted to the long-run average criterion. We apply the long-run average PCL framework of \citet{nmwp0726}, using its verification and discounted-to-average transfer theorems \citep[Theorems~3.6 and~4.5]{nmwp0726}. We use those results as criterion-agnostic tools and check only the model-specific ingredients: average threshold metrics, bias-marginal metrics, marginal-work positivity, index regularity, and PCL identities.

We use ``Abelian'' limits for quantities obtained from discounted metrics as
\(\beta\uparrow1\). For comparison, the ordinary, or non-Abelian, long-run
averages are obtained by fixing a threshold policy first and then taking
\(T\to\infty\) in the Ces\`aro time averages
\[
\frac1T\mathbb E_x^\pi\!\left[\sum_{t=0}^{T-1}R(X(t),A(t))\right],
\qquad
\frac1T\mathbb E_x^\pi\!\left[\sum_{t=0}^{T-1}A(t)\right].
\]
In the present deterministic model on \(\mathsf X^{\mathrm{inv}}\), these ordinary
averages exist and agree with the Abelian limits below. The reason is explicit.
If \(z<p\), the threshold policy is eventually active in every period. If
\(p\leqslant z<z_\infty\), with \(z_{s-1}\leqslant z<z_s\), then after any initial transient the
policy repeats the deterministic pattern
\[
p=z_0 \xrightarrow{0} z_1 \xrightarrow{0}\cdots
\xrightarrow{0} z_s \xrightarrow{1} p,
\]
where \(0\) denotes passivity and \(1\) denotes activation. Thus the ordinary average
reward and work are the reward and work per cycle divided by the cycle length
\(s+1\). If \(z=z_\infty\), the policy is all-passive on the core and the beliefs
converge to \(z_\infty\). The Abelian representation is nevertheless useful because
it connects these ordinary time averages directly to the discounted PCL metrics and
to the bias-level marginal quantities used below.

To keep the notation parallel with the discounted conditions \textup{(PCLI1--PCLI3)}, we label the average marginal-work positivity condition as \textup{(APCLI1)} and the monotonicity and continuity condition for \(m^{\mathrm{avg}}\) as \textup{(APCLI2)}. The discounted LS identity \textup{(PCLI3)} splits in the average framework into two identities: \textup{(APCLI3a)}, an average reward--work identity for threshold gain metrics, and \textup{(APCLI3b)}, a bias-marginal identity used for the average marginal sign characterization. Proposition~\ref{pro:avg-pcl-verification} verifies these average PCL conditions.

For the average criterion the relevant state space is the invariant core
\begin{equation}
\label{eq:avg-core-state-space}
\mathsf X^{\mathrm{inv}}=[0,z_\infty].
\end{equation}
Activation resets to \(p\) and passivity maps \([0,z_\infty]\) into itself. Beliefs \(x>z_\infty\) are transient for average gains and affect only bias-sensitive rankings. The corresponding threshold set is
\begin{equation}
\label{eq:avg-core-threshold-space}
\mathsf Z^{\mathrm{inv}}=(-\infty,z_\infty],
\end{equation}
where \(z<0\) is all-active and \(z=z_\infty\) is all-passive on the core.

For a single patient and price \(\lambda\), define
\begin{equation}
\label{eq:avg-lambda-problem}
J_\lambda(x,\pi)\triangleq
\liminf_{T\to\infty}\frac1T
\mathbb E_x^\pi\!\left[
\sum_{t=0}^{T-1}\{R(X(t),A(t))-\lambda A(t)\}
\right],
\qquad x\in\mathsf X^{\mathrm{inv}},
\end{equation}
and \(J_\lambda^*(x)=\sup_{\pi\in\Pi}J_\lambda(x,\pi)\).

\begin{proposition}[Average-reward optimality equations]
\label{pro:avg-acoe}
For every \(\lambda\in\mathbb R\), the average \(\lambda\)-price problem on \(\mathsf X^{\mathrm{inv}}\) satisfies the average optimality and optimal-action characterization in \citep[Assumption~4.1]{nmwp0726}. Equivalently, there exist a state-independent optimal gain \(\eta_\lambda\) and a continuous bias \(b_\lambda\) satisfying
\begin{equation}
\label{eq:avg-acoe-adherence}
\eta_\lambda+b_\lambda(x)=
\max\{r-\lambda+b_\lambda(p),\ r(1-x)+b_\lambda(h(x))\},
\qquad x\in\mathsf X^{\mathrm{inv}},
\end{equation}
and stationary deterministic maximizers are average-optimal.
\end{proposition}
The proof in Appendix~D verifies the compact-core Lipschitz and vanishing-discount prerequisites in \citet[Lemma~4.2]{nmwp0726}.

\subsection{Average reward and work under threshold policies}
\label{s:lrarwlda}

Let \(F_\beta(x,z)\) and \(G_\beta(x,z)\) be the discounted threshold metrics of Section~\ref{s:ewrm}. For \(x\in\mathsf X^{\mathrm{inv}}\) and \(z\in\mathsf Z^{\mathrm{inv}}\), the long-run average rates are the state-independent Abelian limits
\[
F^{\mathrm{avg}}(z)=\lim_{\beta\nearrow1}(1-\beta)F_\beta(x,z),
\qquad
G^{\mathrm{avg}}(z)=\lim_{\beta\nearrow1}(1-\beta)G_\beta(x,z).
\]
Define the undiscounted passive sum (cf.\ \eqref{eq:Phi-t-p-sum})
\begin{equation}
\label{eq:undPhi-t-p-sum}
\bar\Phi_t(x)\triangleq\sum_{j=0}^{t-1}(1-h_j(x))
=t(1-z_\infty)+\frac{1-\rho^t}{1-\rho}(z_\infty-x).
\end{equation}

\begin{proposition}[Average threshold metrics on the core]
\label{pro:avg-rwm}
For \(z\in\mathsf Z^{\mathrm{inv}}\),
\begin{enumerate}[label=(\alph*)]
\item if \(z<p\), then \(F^{\mathrm{avg}}(z)=r\) and \(G^{\mathrm{avg}}(z)=1\);
\item if \(p\leqslant z<z_\infty\) and \(z_{s-1}\leqslant z<z_s\), then
\[
F^{\mathrm{avg}}(z)=r\frac{1+\bar\Phi_s(p)}{s+1},
\qquad
G^{\mathrm{avg}}(z)=\frac1{s+1};
\]
\item if \(z=z_\infty\), then \(F^{\mathrm{avg}}(z_\infty)=r(1-z_\infty)\) and \(G^{\mathrm{avg}}(z_\infty)=0\).
\end{enumerate}
\end{proposition}
The proof in Appendix~D gives the Abelian/regenerative calculation.

\subsection{Average marginal metrics and MP index on the core}
\label{s:lramrwa}

Let \(f_\beta\) and \(g_\beta\) be the discounted marginal metrics in Section~\ref{s:edmrwm}. Their average counterparts are the finite Abelian bias limits
\[
f^{\mathrm{avg}}(x,z)=\lim_{\beta\nearrow1}f_\beta(x,z),
\qquad
g^{\mathrm{avg}}(x,z)=\lim_{\beta\nearrow1}g_\beta(x,z),
\]
for \(x\in\mathsf X^{\mathrm{inv}}\), \(z\in\mathsf Z^{\mathrm{inv}}\), as in \citet[Sec.~4.2]{nmwp0726}.

The marginal quantities \(f^{\mathrm{avg}}\) and \(g^{\mathrm{avg}}\) are
bias-level quantities, not ordinary differences of long-run average rates between the
two one-step-deviation policies \(\langle1,z\rangle\) and \(\langle0,z\rangle\).
After the initial action, both policies follow the same \(z\)-threshold rule. Hence,
for \(z<p\), both continuations eventually become all-active; for
\(p\leqslant z<z_\infty\), both continuations eventually enter the same deterministic
threshold pattern displayed above; and for \(z=z_\infty\), both continuations are
passive and converge to \(z_\infty\). Therefore their ordinary long-run average
reward and work rates are the same, so subtracting ordinary average rates would lose
the transient effect of the initial action.

The unscaled Abelian limits retain that transient effect. For fixed \(z\), the discounted
cumulative metrics have expansions of the form
\[
F_\beta(x,z)
=
\frac{F^{\mathrm{avg}}(z)}{1-\beta}
+
B_F(x,z)+o(1),
\qquad
G_\beta(x,z)
=
\frac{G^{\mathrm{avg}}(z)}{1-\beta}
+
B_G(x,z)+o(1),
\]
where the leading \(1/(1-\beta)\) terms are the discounted cumulative reward and work
contributions associated with the common long-run average rates, and \(B_F,B_G\)
are finite bias terms. In the one-step comparisons defining \(f_\beta\) and \(g_\beta\),
the common leading cumulative reward/work terms cancel, leaving a finite difference
of bias terms. This is why \(F^{\mathrm{avg}}\) and \(G^{\mathrm{avg}}\) are defined by
\((1-\beta)\)-scaled Abelian limits, whereas \(f^{\mathrm{avg}}\) and \(g^{\mathrm{avg}}\)
are defined by unscaled Abelian limits.

The following result gives closed-form expressions for the average marginal metrics
\(f^{\mathrm{avg}}\) and \(g^{\mathrm{avg}}\).

\begin{proposition}[Average marginal reward and work]
\label{pro:avg-mrwm}
For \(x\in\mathsf X^{\mathrm{inv}}\) and \(z\in\mathsf Z^{\mathrm{inv}}\),
\begin{enumerate}[label=(\alph*)]
\item if \(z<p\), then \(f^{\mathrm{avg}}(x,z)=rx\) and \(g^{\mathrm{avg}}(x,z)=1\);
\item if \(p\leqslant z<z_\infty\), \(z_{s-1}\leqslant z<z_s\), then
\begin{enumerate}[label=(\roman*)]
\item for \(x\leqslant z\) and \(t=\tau(x,z)\),
\[
f^{\mathrm{avg}}(x,z)=r\left[\frac{t}{s+1}\{1+\bar\Phi_s(p)\}-\bar\Phi_t(x)\right],
\qquad
g^{\mathrm{avg}}(x,z)=\frac{t}{s+1};
\]
\item for \(x>z\), \(f^{\mathrm{avg}}(x,z)=r(x-1)+F^{\mathrm{avg}}(z)\) and \(g^{\mathrm{avg}}(x,z)=1/(s+1)\);
\end{enumerate}
\item if \(z=z_\infty\), then \(f^{\mathrm{avg}}(x,z_\infty)=rx/(1-\rho)\) and \(g^{\mathrm{avg}}(x,z_\infty)=1\).
\end{enumerate}
In particular,
\begin{equation}
\label{eq:avg-apcli1-core}
g^{\mathrm{avg}}(x,z)>0,
\qquad x\in\mathsf X^{\mathrm{inv}},\ z\in\mathsf Z^{\mathrm{inv}}.
\end{equation}
\end{proposition}
The proof is in Appendix~D.

\begin{remark}[Why the core restriction is needed]
\label{rem:avg-core-zero-branch}
On the full interval \([0,1]\), thresholds \(z_\infty\leqslant z<1\) have an off-core branch \(x>z\) with \(h(x)>z\), where the Abelian marginal-work limit is zero: activating now and waiting one period before activation produce the same eventual intervention shifted by one period. Thus the full interval violates the marginal-resource positivity condition in \citet[Definition~4]{nmwp0726}; on the invariant core this branch is absent and \eqref{eq:avg-apcli1-core} holds.
\end{remark}

Define \(m^{\mathrm{avg}}(x,z)=f^{\mathrm{avg}}(x,z)/g^{\mathrm{avg}}(x,z)\) and \(m^{\mathrm{avg}}(x)=m^{\mathrm{avg}}(x,x)\) on the core.

\begin{proposition}[Average MP index on the core]
\label{pro:avg-mpi-explicit}
The average MP index on \(\mathsf X^{\mathrm{inv}}\) is
\[
m^{\mathrm{avg}}(x)=
\begin{cases}
rx, & 0\leqslant x<p,\\[5pt]
r\bigl[(t+1)x+\bar\Phi_t(p)-t\bigr],
  & z_{t-1}\leqslant x<z_t,
    \quad t=1,2,\ldots,\\[7pt]
\dfrac{r z_\infty}{1-\rho}, & x=z_\infty.
\end{cases}
\]
Moreover, \(m^{\mathrm{avg}}\) is continuous and strictly increasing on \(\mathsf X^{\mathrm{inv}}\).
\end{proposition}
The proof is in Appendix~D, and hence \textup{(APCLI2)} in \citet[Definition~4]{nmwp0726} holds with no exceptional core states.

\begin{remark}[Transient-state extension]
\label{rem:avg-transient-extension}
The PCL verification is restricted to the core. For transient initial beliefs \(x>z_\infty\), the bias-sensitive vanishing-discount extension is \(\widetilde m^{\mathrm{avg}}(x)=\lim_{\beta\nearrow1}m_\beta(x)=rx/(1-\rho)\), \(z_\infty\leqslant x\le1\), which agrees with \(m^{\mathrm{avg}}\) at \(z_\infty\).
\end{remark}

\subsection{Long-run average PCL-indexability verification}
\label{s:avg-pcl-conditions}

In the following LS identities, \(\mathrm{d}G^{\mathrm{avg}}(z)\) is generated by the right-continuous step function \(G^{\mathrm{avg}}\), and \(\mathrm{d}m^{\mathrm{avg}}(z)\) by the continuous increasing function \(m^{\mathrm{avg}}\), extended outside the core by endpoint constants; see \citet[Ch.~X]{doob94}.

\begin{proposition}[Average PCL-indexability conditions]
\label{pro:avg-pcl-verification}
On \(\mathsf X^{\mathrm{inv}}\), with threshold set \(\mathsf Z^{\mathrm{inv}}\), the following average PCL conditions hold, with no exceptional core states.

\smallskip\noindent\textup{(APCLI1)} Marginal-work positivity holds: \(g^{\mathrm{avg}}(x,z)>0\) for \(x\in\mathsf X^{\mathrm{inv}}\) and \(z\in\mathsf Z^{\mathrm{inv}}\), as stated in \eqref{eq:avg-apcli1-core}.

\smallskip\noindent\textup{(APCLI2)} The average MP index \(m^{\mathrm{avg}}\) is continuous and strictly increasing on \(\mathsf X^{\mathrm{inv}}\), by Proposition~\ref{pro:avg-mpi-explicit}.

\smallskip\noindent\textup{(APCLI3a)} The average reward--work LS identity holds: for finite \(z_1<z_2\leqslant z_\infty\),
\begin{equation}
\label{eq:avg-apcli3a}
F^{\mathrm{avg}}(z_2)-F^{\mathrm{avg}}(z_1)
=
\int_{(z_1,z_2]}m^{\mathrm{avg}}(z)\,\mathrm{d}G^{\mathrm{avg}}(z).
\end{equation}

\smallskip\noindent\textup{(APCLI3b)} The bias-marginal LS identity holds: for \(x,z\in\mathsf X^{\mathrm{inv}}\), with
\(D_x(z)=f^{\mathrm{avg}}(x,z)-m^{\mathrm{avg}}(z)g^{\mathrm{avg}}(x,z)\),
\begin{equation}
\label{eq:avg-apcli3b}
D_x(z)=
\begin{cases}
\displaystyle \int_{(z,x)}g^{\mathrm{avg}}(x,y)\,\mathrm{d}m^{\mathrm{avg}}(y), & z<x,\\[2ex]
\displaystyle -\int_{[x,z]}g^{\mathrm{avg}}(x,y)\,\mathrm{d}m^{\mathrm{avg}}(y), & x\leqslant z.
\end{cases}
\end{equation}

Consequently, the single-patient long-run average model on \(\mathsf X^{\mathrm{inv}}\) is PCL-indexable with respect to threshold policies in the sense of \citet[Definition~4]{nmwp0726}. Moreover, the average marginal sign characterization is
\begin{equation}
\label{eq:avg-marginal-sign}
\sgn\!\left(f^{\mathrm{avg}}(x,z)-m^{\mathrm{avg}}(z)g^{\mathrm{avg}}(x,z)\right)
=
\sgn\!\left(m^{\mathrm{avg}}(x)-m^{\mathrm{avg}}(z)\right),
\qquad x,z\in\mathsf X^{\mathrm{inv}}.
\end{equation}
\end{proposition}
The proof in Appendix~D verifies \textup{(APCLI3a)} and \textup{(APCLI3b)} from the step-function and Abelian-limit formulas; the sign identity follows from \citet[Lemma~3.4]{nmwp0726}.

\begin{theorem}[Average-criterion threshold-indexability on the invariant core]
\label{the:avg-indexability}
Under the long-run average criterion, the single-patient adherence project restricted to \(\mathsf X^{\mathrm{inv}}=[0,z_\infty]\) is threshold-indexable. Its Whittle index on the core is \(m^{\mathrm{avg}}(x)\).
\end{theorem}
The proof, in Appendix~D, applies \citet[Theorem~4.5]{nmwp0726} using Propositions~\ref{pro:avg-acoe}--\ref{pro:avg-pcl-verification} and the discounted PCL-indexability result in Theorem~\ref{the:adherence-indexability}.

}

\rev{
\section{Numerical study}
\label{s:numerics}

This section assesses Whittle's index policy against five benchmarks and uses the Lagrangian dual bound to quantify optimality gaps. We compare Whittle, myopic, round-robin, Lagrangian index, forced-capacity Lagrangian index, and passive policies. Whittle uses \(w_n=m_n\); myopic ranks by \(r_nx\); round-robin cycles through patients; passive never intervenes. The standard Lagrangian rule activates positive \(d_n^*\)'s up to capacity, while the forced-capacity rule activates the \(M\) largest \(d_n^*\)'s. Both use the PCL marginal metrics and \(\lambda^*\) from Sections~\ref{s:iwip} and~\ref{s:bisection}.

For each policy \(\pi\), initial beliefs are i.i.d. \(\mathrm{Unif}[0,1]\), and performance is reported as
\[
\bar V^\pi\triangleq\frac{1-\beta}{N}\,
\mathbb E^\pi\!\left[\sum_{t=0}^{\infty}\sum_{n=1}^N R_n(X_n(t),A_n(t))\beta^t\right].
\]
Policy values use 300 Monte Carlo runs and truncation horizon \(T=700\). With \(\beta=0.99\), the omitted tail factor is \(\beta^{700}\approx8.8\times10^{-4}\), so relative gaps are not materially driven by truncation. Let \(D=\inf_{\lambda\geqslant 0}L(\lambda)\) be the population-averaged dual bound computed as in Section~\ref{s:dual-bound}, set \(\bar D=(1-\beta)D/N\), and define
\[
\gamma^\pi\triangleq\frac{\bar D-\bar V^\pi}{\bar D},
\]
which upper-bounds the true relative optimality gap.

\subsection{Experimental design}

We use three experiment families. The first two are low-\(q\) stress regimes, where spontaneous recovery is limited and allocation matters most; the third adds continuous within-cell heterogeneity.

\paragraph{Experiment 1: two-type stress regime.}
We use four low-\(q\) \(A/B\) population specifications. In each, \((p^A,q^A)\!-
\!(p^B,q^B)\) gives the transition parameters for the more adherence-persistent type \(A\) and the more lapse-prone type \(B\):
\[
\begin{array}{@{}ll@{}}
(0.02,0.03)\!-\!(0.20,0.01), &
(0.05,0.03)\!-\!(0.30,0.01),\\[2pt]
(0.05,0.01)\!-\!(0.35,0.01), &
(0.10,0.03)\!-\!(0.35,0.01).
\end{array}
\]
We vary \(\mathrm{prop}_A\in\{0.5,0.7,0.9\}\), \((r^A,r^B)\in\{(1,1),(1,2),(2,1),(2,2)\}\), \(M/N\in\{0.05,0.10,0.20\}\), and \(N\in\{100,200,400,700,1000\}\), giving \(4\times3\times4\times3\times5=720\) instances.

\paragraph{Experiment 2: three-type stress regime.}
We use three ordered low-\(q\) \(A/B/C\) specifications; each row gives \((p^A,q^A),(p^B,q^B),(p^C,q^C)\), from more adherence-persistent to more lapse-prone:
\[
\begin{array}{ccc}
(0.02,0.03) & (0.10,0.02) & (0.25,0.01)\\
(0.05,0.03) & (0.15,0.02) & (0.30,0.01)\\
(0.05,0.01) & (0.15,0.01) & (0.35,0.01).
\end{array}
\]
We vary the \(A/B/C\) share vector over \((0.7,0.2,0.1)\), \((0.6,0.3,0.1)\), and \((0.5,0.3,0.2)\); the reward vector over \((r^A,r^B,r^C)\in\{(1,1,1),(1,1,2),(1,2,1),(1,2,2)\}\); and again \(M/N\) and \(N\) over the above grids. This gives \(3\times3\times4\times3\times5=540\) instances.

\paragraph{Experiment 3: jittered finite-mixture population.}
We use three dynamic archetypes \(A,B,C\) with base parameters \((\bar p^A,\bar q^A)=(0.03,0.03)\), \((\bar p^B,\bar q^B)=(0.10,0.02)\), and \((\bar p^C,\bar q^C)=(0.25,0.01)\). Crossing them with \(r\in\{1,2\}\) gives cells \((A,1),(A,2),(B,1),(B,2),(C,1),(C,2)\), with baseline and stress probabilities
\[
\omega^{\mathrm{base}}=(0.35,0.15,0.245,0.105,0.105,0.045),
\omega^{\mathrm{stress}}=(0.25,0.10,0.25,0.15,0.15,0.10).
\]
For each patient, we draw a cell, fix \(r_n\) by that cell, and if the selected dynamic archetype is \(k\), draw independently conditional on the cell
\[
p_n\sim\mathrm{Unif}[\bar p^k-0.01,\bar p^k+0.01],
\qquad
q_n\sim\mathrm{Unif}[\bar q^k-0.005,\bar q^k+0.005],
\]
with rejection-resampling to enforce \(p_n+q_n\le0.95\). Thus the \(A/B/C\) labels are centers of adherence dynamics, not exact patient types. We use the two mixture laws, ten seeded population draws for each law and \(N\), reuse each draw across \(M/N\in\{0.05,0.10,0.20\}\), and vary \(N\in\{100,200,400,700,1000\}\), giving \(2\times10\times3\times5=300\) instances.

\begin{table}[!htbp]
  \centering
  \caption{Summary of the three experiment families.}
  \label{tab:numerics-design}
  \small
  \setlength{\tabcolsep}{3pt}
  \begin{tabular}{@{}p{0.22\textwidth}p{0.42\textwidth}c p{0.24\textwidth}@{}}
    \hline
    Experiment & Heterogeneity model & Cases & Purpose \\
    \hline
    Two-type stress & Four low-\(q\) \(A/B\) specifications; four reward profiles & 720 & Mechanism and stress test \\
    Three-type stress & Three ordered low-\(q\) \(A/B/C\) specifications; four reward profiles & 540 & Richer discrete heterogeneity \\
    Jittered mixture & Three \(A/B/C\) dynamic archetypes crossed with \(r\in\{1,2\}\); two mixture laws; jittered \((p,q)\) & 300 & Continuous within-cell heterogeneity \\
    \hline
  \end{tabular}
\end{table}

Unless otherwise stated, each reported mean averages over all instances in the
corresponding experiment family, including the full population-size grid
\(N\in\{100,200,400,700,1000\}\). When a table conditions on one design factor,
such as \(M/N\), reward profile, composition, or mixture law, the mean is taken
over the remaining design factors, including \(N\). The \(N\)-grid is included as
a robustness check across cohort sizes, not as a separate scaling experiment.

The Online Supplement reports the full gap and reward-ratio distributions, worst myopic cases, and disaggregated capacity, reward-profile, composition, and mixture-law summaries.

\subsection{Summary results across experiment families}

Table~\ref{tab:gaps-three-experiments} reports mean bound-relative Lagrangian gaps. Whittle and forced-capacity Lagrangian index policies are the two best-performing policies, with mean gaps at most \(0.43\%\) and \(0.40\%\), respectively. The standard Lagrangian rule is weaker but remains well ahead of myopic prioritization; passive performs worst. Tables~1--3 of the Online Supplement give the full gap distributions and show maximum Whittle gaps of \(1.34\%\), \(1.09\%\), and \(0.61\%\) in Experiments 1--3.

\begin{table}[!htbp]
  \centering
  \caption{Mean relative Lagrangian gaps \(\gamma^\pi\) (in \%) by experiment family.}
  \label{tab:gaps-three-experiments}
  \begin{tabular}{lcccccc}
    \hline
    Experiment & Whittle & Myopic & Round-robin & Lagrangian & Forced Lag. & Passive \\
    \hline
    Two-type stress      & 0.403 & 10.100 & 3.586 & 1.950 & 0.388 & 61.635 \\
    Three-type stress    & 0.430 & 12.243 & 5.233 & 1.748 & 0.395 & 63.801 \\
    Jittered mixture     & 0.386 &  9.053 & 4.943 & 1.284 & 0.361 & 58.520 \\
    \hline
  \end{tabular}
\end{table}

For compactness, the ratio tables use \(\mathrm{Wh}\), \(\mathrm{my}\), \(\mathrm{Lag}\), and \(\mathrm{LagF}\) for Whittle, myopic, Lagrangian, and forced-capacity Lagrangian index policies. Table~\ref{tab:ratios-three-experiments} shows that Whittle improves mean reward over myopic by about \(12.4\%\), \(14.6\%\), and \(9.7\%\) in the three designs. Forced-capacity Lagrangian indexing is essentially tied with Whittle, while the standard Lagrangian rule trails Whittle modestly but systematically. Tables~4--6 of the Online Supplement give the corresponding ratio distributions.

\begin{table}[!htbp]
  \centering
  \caption{Mean and maximum reward ratios by experiment family.}
  \label{tab:ratios-three-experiments}
  \begin{tabular}{lcccc}
    \hline
    Experiment & \(\bar V^{\mathrm{Wh}}/\bar V^{\mathrm{my}}\) &
    \(\bar V^{\mathrm{Lag}}/\bar V^{\mathrm{Wh}}\) &
    \(\bar V^{\mathrm{LagF}}/\bar V^{\mathrm{Wh}}\) &
    \(\max \bar V^{\mathrm{Wh}}/\bar V^{\mathrm{my}}\) \\
    \hline
    Two-type stress      & 1.124 & 0.9845 & 1.0001 & 1.787 \\
    Three-type stress    & 1.146 & 0.9868 & 1.0004 & 1.690 \\
    Jittered mixture     & 1.097 & 0.9910 & 1.0002 & 1.166 \\
    \hline
  \end{tabular}
\end{table}

The largest Whittle gains occur in the curated stress designs, where myopic priorities can be severely misaligned, and the gains persist under jittered mixtures. Table~7 of the Online Supplement and its accompanying discussion document the ten worst myopic cases in Experiment~1; they concentrate on the \((0.05,0.01)\)-\((0.35,0.01)\) structural regime with balanced populations and \(M/N=0.10\), where myopic loses roughly one third of Whittle's reward.

\subsection{Where the Whittle index policy's advantage is strongest}

Table~\ref{tab:capacity-ratio-summary} shows the mean ratio \(\bar V^{\mathrm{Wh}}/\bar V^{\mathrm{my}}\) by capacity. In every experiment family, the ratio is largest at \(M/N=0.05\), smaller at \(0.10\), and smaller still at \(0.20\). Thus dynamic index-based prioritization matters more when intervention resources are scarcer; Online Supplement Tables~8--10 give the family-specific capacity breakdowns.

\begin{table}[!htbp]
  \centering
  \caption{Mean reward ratio \(\bar V^{\mathrm{Wh}}/\bar V^{\mathrm{my}}\) by capacity ratio.}
  \label{tab:capacity-ratio-summary}
  \begin{tabular}{lccc}
    \hline
    Experiment & \(M/N=0.05\) & \(M/N=0.10\) & \(M/N=0.20\) \\
    \hline
    Two-type stress      & 1.193 & 1.132 & 1.047 \\
    Three-type stress    & 1.211 & 1.163 & 1.063 \\
    Jittered mixture     & 1.132 & 1.110 & 1.048 \\
    \hline
  \end{tabular}
\end{table}

Reward heterogeneity in Experiment~1 is also informative. Online Supplement Table~11 gives reward-profile ratios: \(1.021\) for \((r^A,r^B)=(2,1)\), versus \(1.164\), \(1.149\), and \(1.164\) for \((1,1)\), \((1,2)\), and \((2,2)\). Recall that the myopic rule assigns priority \(r_n x\), the immediate expected reward gain from activation, whereas Whittle's policy assigns priority \(m_n(x)\), which accounts for future adherence dynamics. The smaller ratio for \((2,1)\) suggests that giving the larger reward to the more adherence-persistent type makes the myopic priority closer to the dynamic index priority; otherwise, myopic prioritization remains substantially worse. Online Supplement Table~12 shows that, in Experiment~2, the mean ratio
\(\bar V^{\mathrm{Wh}}/\bar V^{\mathrm{my}}\) remains similar across the three
population-share vectors, taking values \(1.148\), \(1.136\), and \(1.154\).
Online Supplement Table~13 gives the corresponding mixture-law breakdown for
Experiment~3, with mean ratios \(1.087\) under the baseline law and \(1.106\)
under the stress law. Thus the Whittle improvement over myopic prioritization
is not driven by a single three-type composition or by only one of the two
jittered-mixture laws.

}

\section{Conclusions}
\label{s:concl}

We developed a belief-state restless bandit model for capacity-constrained adherence outreach with reset-type interventions. Using the PCL approach in \cite{nmmor20}, we proved discounted threshold-indexability, derived a closed-form Whittle index, and characterized how priorities vary with lapse and spontaneous-recovery parameters.

\rev{
We also established the corresponding single-patient long-run average threshold-indexability result on the invariant core \(\mathsf X^{\mathrm{inv}}=[0,z_\infty]\). Using the PCL-based discounted-to-average transfer framework in \cite{nmwp0726}, we obtained an explicit average-criterion Whittle index and clarified why off-core transient states affect bias-sensitive rankings but not recurrent average reward/work rates.}

\rev{
The same PCL marginal metrics yielded a piecewise-affine single-patient Lagrangian value and computable Lagrangian index benchmarks. They also give a dual-bound computation running in \(O(N\log(1/\varepsilon))\) time.}

\rev{
Numerically, Whittle indexing substantially improves on myopic prioritization in curated low-\(q\) stress regimes, especially under tight capacity, and the gain persists under jittered finite-mixture heterogeneity. The forced-capacity Lagrangian index policy is nearly indistinguishable from Whittle's policy in value, while the standard Lagrangian index policy trails Whittle modestly but remains well ahead of myopic prioritization.}

Together, the explicit discounted and average-criterion indices, parameter analysis, computable bounds, and PCL-derived benchmarks support the design and evaluation of index policies in belief-state RMABs.

Future work includes extensions beyond the present reset-type model, such as richer latent adherence dynamics, partial observations under passivity, empirical calibration of transition and reward parameters, and additional objectives or constraints involving fairness, risk, or equity.

\section*{Statements and Declarations}

\subsection*{Funding}
This research did not receive any specific grant from funding agencies in the public, commercial, or not-for-profit sectors.
The first author's work was supported by an internal UC3M research grant (\emph{Programa Propio de Investigaci\'on}).

\subsection*{Competing Interests}
The authors declare that they have no competing interests.

\subsection*{Data availability}
No external data were used in this study. All numerical results are generated by
the simulation procedures described in the paper from specified parameter grids and
pseudo-random initial beliefs. The Online Supplement reports the full numerical
summaries used in the paper. The underlying simulation-output files are available
from the corresponding author upon request.


\clearpage
\section*{Online Supplement}
\begin{center}
{\large for ``A belief-state restless bandit model for treatment adherence: Whittle indexability via partial conservation laws''}
\end{center}
\vspace{1em}

\appendix
\setcounter{section}{0}
\setcounter{subsection}{0}
\renewcommand{\thesection}{\Alph{section}}
\renewcommand{\thesubsection}{\Alph{section}.\arabic{subsection}}

\setcounter{equation}{0}
\setcounter{table}{0}
\setcounter{figure}{0}
\setcounter{lemma}{0}
\setcounter{corollary}{0}
\renewcommand{\theequation}{S\arabic{equation}}

\providecommand{\theHsection}{}
\providecommand{\theHsubsection}{}
\providecommand{\theHequation}{}
\providecommand{\theHtable}{}
\providecommand{\theHfigure}{}
\providecommand{\theHlemma}{}
\providecommand{\theHcorollary}{}
\renewcommand{\theHsection}{supp.\Alph{section}}
\renewcommand{\theHsubsection}{supp.\Alph{section}.\arabic{subsection}}
\renewcommand{\theHequation}{supp.\arabic{equation}}
\renewcommand{\theHtable}{supp.\arabic{table}}
\renewcommand{\theHfigure}{supp.\arabic{figure}}
\renewcommand{\theHlemma}{supp.\arabic{lemma}}
\renewcommand{\theHcorollary}{supp.\arabic{corollary}}

\section{Proofs for Section 5 (Verification of PCL-indexability conditions)}
\label{app:avpclic}

\subsection{Proofs for Section 5.1 (Discounted reward and work metrics under threshold policies)}
\label{app:ewrm}

\begin{proof}[Proof of Proposition~\ref{pro:edrwm}]
\textup{(a) Case $z<p$.}
Fix $z<p$.

\medskip\noindent
\emph{Case 1: $0\leqslant x\leqslant z$.}
Since $h$ is increasing and $h(0)=p>z$, we have $h(x)\geqslant p>z$. Thus the policy is passive at $t=0$ and active from
$t=1$ onward. Using the functional equations \eqref{eq:Fxz}--\eqref{eq:Gxz} and the fact that whenever the policy is
active it resets the next belief to $p$, we obtain
\begin{align*}
F(x,z)& =r(1-x)+\beta F\bigl(h(x),z\bigr)
      =r(1-x)+\beta\bigl[r+\beta F(p,z)\bigr] \\
      & =r(1-x+\beta)+\beta^2F(p,z),
\end{align*}
and
\[
G(x,z)=\beta G\bigl(h(x),z\bigr)
      =\beta\bigl[1+\beta G(p,z)\bigr]
      =\beta+\beta^2G(p,z).
\]

\medskip\noindent
\emph{Case 2: $x>z$.}
The policy intervenes immediately at $t=0$, hence by \eqref{eq:Fxz}--\eqref{eq:Gxz}
\[
F(x,z)=r+\beta F(p,z),\quad
G(x,z)=1+\beta G(p,z).
\]

\medskip\noindent
\emph{Solving for $F(p,z)$ and $G(p,z)$.}
Since $p>z$, the point $x=p$ falls into Case 2, so
\[
F(p,z)=r+\beta F(p,z),\quad
G(p,z)=1+\beta G(p,z),
\]
which yields
\[
F(p,z)=\frac{r}{1-\beta},\quad
G(p,z)=\frac{1}{1-\beta}.
\]
Substituting these into the formulas in Cases 1--2 gives the stated expressions for $F(x,z)$ and $G(x,z)$ when $z<p$.

\medskip\noindent
\textup{(b) Case $p\leqslant z<z_\infty$.}
Assume $p\leqslant z<z_\infty$ and let $t=t(z)\in\{1,2,\ldots\}$ be the unique index such that
\[
z_{t-1}\leqslant z<z_t,
\quad\text{equivalently}\quad t=\tau(p,z),
\]
where $\tau(\cdot,\cdot)$ is the first crossing time defined in \eqref{eq:tauxz}.

If $x\leqslant z$, then the belief evolves passively along $h_s(x)$ until it first exceeds $z$, i.e., until time
$\tau(x,z)\in\{1,2,\ldots\}$, characterized by \eqref{eq:tau-characterization}. Up to time $\tau(x,z)-1$ the policy is
passive, and at time $\tau(x,z)$ it intervenes; thereafter, the post-intervention belief is $p$ and the continuation
value is the same as if the process started from $p$ at the next epoch. Using \eqref{eq:Fxz}--\eqref{eq:Gxz} and the
definitions of $\Phi_s(\cdot)$, we obtain for all $x\leqslant z$:
\[
G(x,z)=\beta^{\tau(x,z)}\Bigl[1+\beta\,G(p,z)\Bigr],
\quad
F(x,z)=r\,\Phi_{\tau(x,z)}(x)+\beta^{\tau(x,z)}\Bigl[r+\beta\,F(p,z)\Bigr].
\]

If $x>z$, then $\tau(x,z)=0$ (cf.\ \eqref{eq:tauxz}) and the policy intervenes immediately at time $0$, so
\[
F(x,z)=r+\beta F(p,z),\quad
G(x,z)=1+\beta G(p,z).
\]

It remains to determine $F(p,z)$ and $G(p,z)$ in terms of $t=\tau(p,z)$. Since $p\leqslant z$, we are in the case $x\leqslant z$
above, so substituting $x=p$ and $\tau(p,z)=t$ gives
\[
G(p,z)=\beta^{t}\Bigl[1+\beta\,G(p,z)\Bigr],\quad
F(p,z)=r\,\Phi_t(p)+\beta^{t}\Bigl[r+\beta\,F(p,z)\Bigr].
\]
Solving the resulting linear equations yields
\[
G(p,z)=\frac{\beta^{t}}{1-\beta^{t+1}},
\quad
F(p,z)=r\,\frac{\Phi_t(p)+\beta^{t}}{1-\beta^{t+1}}.
\]
Substituting these into the general formulas for $x\leqslant z$ and $x>z$ yields the stated case-(b) expressions.

\medskip\noindent
\textup{(c) Case $z_\infty\leqslant z<1$.}
If $x\leqslant z$, then the passive belief trajectory $h_t(x)$ is increasing and converges to $z_\infty\leqslant z$, hence it never
exceeds the threshold and no interventions ever occur. Therefore,
\[
G(x,z)=0,
\quad
F(x,z)=\sum_{t=0}^\infty r\bigl(1-h_t(x)\bigr)\beta^t
      =r\,\Phi_\infty(x).
\]
If $x>z$, an intervention is taken at $t=0$ and the next belief is $p\leqslant z$, so no further interventions occur:
\[
G(x,z)=1+\beta G(p,z)=1,
\quad
F(x,z)=r+\beta F(p,z)=r\bigl[1+\beta\,\Phi_\infty(p)\bigr].
\]

\medskip\noindent
\textup{(d) Case $z\geqslant 1$.}
For any $x\in[0,1]$ we have $x\leqslant z$, so the threshold is never crossed and the system remains passive forever. Hence
$G(x,z)=0$ and
\[
F(x,z)
=\sum_{t=0}^{\infty} r\bigl(1-h_t(x)\bigr)\beta^t
=r\,\Phi_\infty(x).
\]
This completes the proof.
\end{proof}

\subsection{Proofs for Section 5.2 (Discounted marginal metrics and condition (PCLI1))}
\label{app:dmrwm}

\begin{proof}[Proof of Proposition~\ref{pro:dmrwm}]
For any $z$ and $x$, under the one-step-deviation policies $\langle a,z\rangle$,
\[
F\bigl(x,\langle 1,z\rangle\bigr)=r+\beta F(p,z),
\quad
F\bigl(x,\langle 0,z\rangle\bigr)=r(1-x)+\beta F\bigl(h(x),z\bigr),
\]
\[
G\bigl(x,\langle 1,z\rangle\bigr)=1+\beta G(p,z),
\quad
G\bigl(x,\langle 0,z\rangle\bigr)=\beta G\bigl(h(x),z\bigr),
\]
so
\begin{equation}\label{eq:fg-Qform-simpl}
f(x,z)=rx+\beta\bigl(F(p,z)-F(h(x),z)\bigr),
\quad
g(x,z)=1+\beta\bigl(G(p,z)-G(h(x),z)\bigr).
\end{equation}
We evaluate $F(h(x),z)$ and $G(h(x),z)$ using Proposition~\ref{pro:edrwm}.

\smallskip\noindent
\emph{(a) Case $z<p$.}
Then $h(x)\geqslant p>z$ for all $x$, and Proposition~\ref{pro:edrwm}\textup{(a)} gives
$F(p,z)=F(h(x),z)=r/(1-\beta)$ and $G(p,z)=G(h(x),z)=1/(1-\beta)$.
Substituting into \eqref{eq:fg-Qform-simpl} yields $f(x,z)=rx$ and $g(x,z)=1$.

\smallskip\noindent
\emph{(b) Case $p\leqslant z<z_\infty$.}
Let $t=\tau(p,z)$ and define
\[
K_F(z)\triangleq r+\beta F(p,z),
\quad
K_G(z)\triangleq 1+\beta G(p,z).
\]
For any $y>z$, Proposition~\ref{pro:edrwm}\textup{(b)} yields $F(y,z)=K_F(z)$ and $G(y,z)=K_G(z)$.

\smallskip\noindent
\emph{(b.1) Subcase $x\leqslant z$.}
Let $s=\tau(x,z)$. Since $h$ is increasing and $h_0(x)=x\leqslant z$, we have
\[
\tau\bigl(h(x),z\bigr)=s-1,
\]
because $h_{s-1}(h(x))=h_s(x)>z$ and $h_{s-2}(h(x))=h_{s-1}(x)\leqslant z$ (for $s\geqslant 2$; if $s=1$ then $h(x)>z$ and
$\tau(h(x),z)=0$ as well). Applying Proposition~\ref{pro:edrwm}\textup{(b)} with initial state $h(x)$ and crossing time
$s-1$ gives
\[
F(h(x),z)=r\,\Phi_{s-1}\!\bigl(h(x)\bigr)+\beta^{s-1}K_F(z),
\quad
G(h(x),z)=\beta^{s-1}K_G(z).
\]
Using \eqref{eq:fg-Qform-simpl} and the identity
$r x+r\beta\,\Phi_{s-1}(h(x))=r\,\Phi_s(x)$, we obtain
\[
f(x,z)=(1-\beta^s)K_F(z)-r\,\Phi_s(x),
\quad
g(x,z)=(1-\beta^s)K_G(z),
\]
which yields the $x\leqslant z$ formulas in part \textup{(b)}.

\smallskip\noindent
\emph{(b.2) Subcase $x>z$.}
Since $z<z_\infty$, we have $h(z)>z$ and $h$ is increasing, hence $h(x)>z$.
Thus $F(h(x),z)=K_F(z)$ and $G(h(x),z)=K_G(z)$. Substituting into \eqref{eq:fg-Qform-simpl} gives
\[
f(x,z)=rx+\beta\bigl(F(p,z)-K_F(z)\bigr)=r(x-1)+(1-\beta)\,K_F(z),
\]
\[
g(x,z)=1+\beta\bigl(G(p,z)-K_G(z)\bigr)=(1-\beta)\,K_G(z).
\]

\smallskip\noindent
\emph{(c) Case $z_\infty\leqslant z<1$.}
Here Proposition~\ref{pro:edrwm}\textup{(c)} yields $F(p,z)=r\Phi_\infty(p)$ and $G(p,z)=0$; define also
\[
K_F(z)\triangleq r+\beta F(p,z)=r\bigl[1+\beta\Phi_\infty(p)\bigr].
\]
Moreover, for any $y$,
\[
F(y,z)=
\begin{cases}
r\Phi_\infty(y), & y\leqslant z,\\
K_F(z), & y>z,
\end{cases}
\quad
G(y,z)=
\begin{cases}
0, & y\leqslant z,\\
1, & y>z.
\end{cases}
\]

\smallskip\noindent
\emph{(c.1) Subcase $h(x)\leqslant z$.}
Then $F(h(x),z)=r\Phi_\infty(h(x))$ and $G(h(x),z)=0$, so \eqref{eq:fg-Qform-simpl} gives
\[
g(x,z)=1,
\quad
f(x,z)=rx+\beta r\Bigl[\Phi_\infty(p)-\Phi_\infty\bigl(h(x)\bigr)\Bigr].
\]
Using the closed form
\[
\Phi_\infty(u)=\frac{1-z_\infty}{1-\beta}+\frac{z_\infty-u}{1-\beta\rho},
\quad u\in[0,1],
\]
we have
$\Phi_\infty(p)-\Phi_\infty(h(x))=(h(x)-p)/(1-\beta\rho)=\rho x/(1-\beta\rho)$, and hence
\[
f(x,z)=rx+\beta r\frac{\rho x}{1-\beta\rho}=\frac{r\,x}{1-\beta\rho}.
\]

\smallskip\noindent
\emph{(c.2) Subcase $x>z$ and $h(x)>z$.}
Then $F(h(x),z)=K_F(z)$ and $G(h(x),z)=1$. Since $G(p,z)=0$, \eqref{eq:fg-Qform-simpl} yields
\[
g(x,z)=1-\beta,
\quad
f(x,z)=rx+\beta\bigl(r\Phi_\infty(p)-K_F(z)\bigr)
      =r(x-1)+(1-\beta)\,K_F(z),
\]
using $K_F(z)=r\bigl[1+\beta\Phi_\infty(p)\bigr]$.

\smallskip\noindent
\emph{(d) Case $z\geqslant 1$.}
Then the $z$-threshold policy is always passive, so $F(y,z)=r\Phi_\infty(y)$ and $G(y,z)=0$ for all $y$.
Thus $G(p,z)=G(h(x),z)=0$ and $g(x,z)=1$. Also $F(p,z)=r\Phi_\infty(p)$ and $F(h(x),z)=r\Phi_\infty(h(x))$, so the same
calculation as in \textup{(c.1)} gives
\[
f(x,z)=rx+\beta r\frac{\rho x}{1-\beta\rho}=\frac{r\,x}{1-\beta\rho}.
\]
This completes the proof.
\end{proof}

\subsection{Proofs for Section 5.3 (MP index and condition (PCLI2)) }
\label{app:pcli2}
\begin{proof}[Proof of Lemma~\ref{lma:pcli2}]
We first prove continuity, then strict monotonicity.

\medskip\noindent
\emph{1. Continuity.}
By Proposition~\ref{pro:dmpi-explicit}, $m$ is affine on each open interval
\[
(0,p),\quad (z_{t-1},z_t)\ (t\geqslant 1),\quad (z_\infty,1),
\]
hence continuous there. Thus, continuity can only fail at the breakpoints
\[
x=p,\quad x=z_t\ (t\geqslant 1),\quad x=z_\infty,\quad x=1.
\]

\smallskip\noindent
\emph{(a) Continuity at $x=p$.}
For $0\leqslant x<p$, $m(x)=rx$, so $\lim_{x\nearrow p}m(x)=rp$.
Since $p=z_0$, the middle branch with $t=1$ applies at $x=p$:
\[
m(x)=\frac{r}{1-\beta}\Bigl[(1-\beta^{2})x+\beta^{2}-\beta+\beta(1-\beta)\Phi_1(p)\Bigr],
\quad z_0\leqslant x<z_1,
\]
and $\Phi_1(p)=1-z_0=1-p$. Substituting $x=p$ and $\Phi_1(p)=1-p$ gives $m(p)=rp$,
so $m$ is continuous at $x=p$.

\smallskip\noindent
\emph{(b) Continuity at internal breakpoints $x=z_t$, $t\geqslant 1$.}
For $x\in[z_{t-1},z_t)$ the $t$-th middle branch applies:
\[
m_t(x)\triangleq
\frac{r}{1-\beta}\Bigl[(1-\beta^{t+1})x+\beta^{t+1}-\beta+\beta(1-\beta)\Phi_t(p)\Bigr],
\]
whereas for $x\in[z_t,z_{t+1})$ the $(t+1)$-st branch applies:
\[
m_{t+1}(x)=
\frac{r}{1-\beta}\Bigl[(1-\beta^{t+2})x+\beta^{t+2}-\beta+\beta(1-\beta)\Phi_{t+1}(p)\Bigr].
\]
We show $m_t(z_t)=m_{t+1}(z_t)$. Using
\[
\Phi_{t+1}(p)=\Phi_t(p)+(1-z_t)\beta^t,
\]
and substituting $x=z_t$ into both expressions, a direct simplification yields equality.
Hence neighboring middle branches match at each $z_t$, so $m$ is continuous at $x=z_t$.

\smallskip\noindent
\emph{(c) Continuity at $x=z_\infty$.}
Since $z_t\nearrow z_\infty$, we have
\[
m(z_\infty^-)=\lim_{t\to\infty} m(z_t)=\lim_{t\to\infty} m_t(z_t),
\]
where the last equality uses part (b). From the last branch,
\[
m(z_\infty^+)=\frac{r\,z_\infty}{1-\beta\rho}.
\]
Taking $t\to\infty$ in the explicit expression for $m_t(z_t)$ (using the closed forms for
$z_t$ and $\Phi_t(p)$ and the definition of $\Phi_\infty(\cdot)$) gives
\[
\lim_{t\to\infty} m_t(z_t)=\frac{r\,z_\infty}{1-\beta\rho}.
\]
Therefore $m(z_\infty^-)=m(z_\infty^+)$, and $m$ is continuous at $x=z_\infty$.

\smallskip\noindent
\emph{(d) Continuity at $x=1$.}
For $x$ near $1$ only the last branch applies, $m(x)=r x/(1-\beta\rho)$, which is affine and hence
continuous at $x=1$.

Combining (a)--(d), $m$ is continuous on $[0,1]$.

\medskip\noindent
\emph{2. Strict monotonicity.}
On each interval where $m$ is affine, its slope is positive:
\begin{itemize}[leftmargin=*]
\item On $(0,p)$, $m(x)=rx$, so $m'(x)=r>0$.
\item On $(z_{t-1},z_t)$ for $t\ge1$, the $t$-th middle branch has slope
\[
m'(x)=\frac{r}{1-\beta}\,(1-\beta^{t+1})>0.
\]
\item On $(z_\infty,1)$, $m(x)=r x/(1-\beta\rho)$, so
\[
m'(x)=\frac{r}{1-\beta\rho}>0
\]
since $0<\beta\rho<1$.
\end{itemize}
Thus $m$ is increasing on the interior of each interval
\[
[0,p),\quad [z_{t-1},z_t)\ (t\ge1),\quad [z_\infty,1].
\]
Because these intervals are ordered by increasing $x$, intersect only at endpoints, and $m$ is continuous at every
breakpoint (Part~1), it follows that $m$ is increasing on $[0,1]$. In particular, $m$ is continuous and
nondecreasing, so condition \textup{(PCLI2)} holds.
\end{proof}

\subsection{Proofs for Section 5.4 (Verifying \textup{(PCLI3)} and threshold-indexability)}
\label{app:pcli3}

This appendix provides the details underlying Proposition~\ref{pro:reachable-set}, which is used in
Section~\ref{s:pcli3} to verify \textup{(PCLI3)} via \citet[Proposition~6]{nmmor20}.
For convenience we restate that sufficient condition.

\begin{lemma}[\cite{nmmor20}, Proposition~6, p.~489]\label{lma:prop6}
Suppose that conditions \textup{(PCLI1)} and \textup{(PCLI2)} hold and that, moreover, for each initial state $x$ there exists
a countable set $\mathcal{D}(x)$ which is the set of endpoints of a partition of the state space (here $[0,1]$) into
left-semiclosed intervals or singletons, with the property that the state trajectory $(X_t)_{t\geqslant 0}$ starting from $x$
is contained in $\mathcal{D}(x)$ under any threshold policy. Then condition \textup{(PCLI3)} holds.
\end{lemma}

\paragraph{A partition induced by a countable closed set.}
We first record a general analytic lemma on partitions induced by countable closed subsets of a compact interval.

\begin{lemma}\label{lem:partition-closed-countable}
Let $K\subset[0,1]$ be a countable closed set such that $0,1\in K$.
Then there exists a partition of $[0,1]$ into pairwise disjoint sets of the form $(a,b]$ or $\{c\}$, with $a<b$ and
$a,b,c\in K$, whose union is $[0,1]$. Moreover, $K$ is precisely the set of finite endpoints of those intervals and singletons.
\end{lemma}

\begin{proof}
Set $U\triangleq[0,1]\setminus K$. Since $K$ is closed in $[0,1]$, $U$ is relatively open in $[0,1]$, hence of the form
$U=V\cap[0,1]$ for some open set $V\subset\mathbb{R}$. By the standard structure theorem for open sets in $\mathbb{R}$
(see, e.g., \citet[Prop.~4.8]{folland99}), we can write
\[
V=\bigcup_{n\geqslant 1} (a_n,b_n)
\]
as a countable disjoint union of open intervals $(a_n,b_n)$, and hence
\[
U=\bigcup_{n\geqslant 1} (a_n,b_n)\cap[0,1]=\bigsqcup_{n\geqslant 1} (a_n,b_n),
\]
where we have relabeled so that each $(a_n,b_n)\subset[0,1]$.

Since $K$ is closed in $[0,1]$, each $a_n$ and $b_n$ belongs to $[0,1]\setminus U = K$.
Now define
\[
\mathcal{P}
\;\triangleq\;
\{(a_n,b_n] : n\geqslant 1\}\;\cup\;\bigl\{\{c\}: c\in K \text{ and } c\neq a_n \text{ for all }n\bigr\}.
\]
The blocks in $\mathcal{P}$ are left-semiclosed, pairwise disjoint, and their union is $[0,1]$ (each $a_n$ is absorbed into
the corresponding $(a_n,b_n]$, and all remaining points of $K$ appear as singletons). By construction, the finite endpoints of
all blocks lie in $K$, and conversely every $c\in K$ is an endpoint of some block. Hence $\mathcal{P}$ is the required
partition and $K$ is precisely its endpoint set.
\end{proof}

\paragraph{Reachable sets under threshold policies.}
For $x\in[0,1]$ and a threshold $z\in[0,1]$, let $\mathcal{D}(x,z)$ denote the set of belief states visited by the trajectory
$(X_t)_{t\geqslant 0}$ starting from $X_0=x$ under the $z$-threshold policy. Define the global reachable set under all
threshold policies by
\[
\mathcal{D}_0(x)\;\triangleq\;\bigcup_{z\in[0,1]} \mathcal{D}(x,z),
\]
and set
\[
\mathcal{D}(x)\;\triangleq\;\overline{\mathcal{D}_0(x)}\cup\{0,1\}\subset[0,1],
\]
where the closure is taken in $[0,1]$.

\begin{proof}[Proof of Proposition~\ref{pro:reachable-set}]
Fix $z\in[0,1]$. Under the $z$-threshold policy, whenever $A_t=0$ (passive) we have $X_{t+1}=h(X_t)$, so along any maximal
passive spell starting from $u$ the belief evolves as $X_{s+k}=h_k(u)$, $k=0,1,\dots$; whenever $A_t=1$ (active), the belief
jumps deterministically to $p$.

Starting from $X_0=x$, any trajectory under the $z$-threshold policy therefore consists of:
\begin{itemize}[leftmargin=*]
\item an initial passive spell from $x$ (possibly of length zero if $x>z$), during which $X_t=h_t(x)$; and
\item thereafter, a sequence of cycles in which an active step takes the state to $p$, followed by a passive spell from $p$,
  during which the trajectory visits $h_t(p)=z_t$.
\end{itemize}
Hence every visited state lies in $\{h_t(x):t\geqslant 0\}\cup\{z_t:t\geqslant 0\}$, i.e.,
\begin{equation}\label{eq:Dxz-inclusion-calD}
\mathcal{D}(x,z)\;\subset\;\{h_t(x):t\geqslant 0\}\;\cup\;\{z_t:t\geqslant 0\}
\quad\text{for all }z\in[0,1].
\end{equation}
Taking the union over $z$ yields
\begin{equation}\label{eq:D0x-inclusion-calD}
\mathcal{D}_0(x)
=\bigcup_{z\in[0,1]} \mathcal{D}(x,z)
\;\subset\;\{h_t(x):t\geqslant 0\}\;\cup\;\{z_t:t\geqslant 0\}.
\end{equation}

Conversely, let $t\geqslant 0$ be arbitrary.
\begin{itemize}[leftmargin=*]
\item To realize $h_t(x)$, choose $z$ strictly larger than $\max\{h_s(x):0\leqslant s<t\}$. Then the $z$-threshold policy is
  passive for the first $t$ steps, so $X_t=h_t(x)$ and hence $h_t(x)\in \mathcal{D}(x,z)\subset \mathcal{D}_0(x)$.
\item To realize $z_t=h_t(p)$, choose any $z$ for which an intervention occurs at some finite time (e.g., any $z<p$ or any
  $p\leqslant z<z_\infty$). Then the process eventually jumps to $p$ and subsequently evolves passively from $p$, visiting
  $z_s$, $s=0,1,2,\dots$, and in particular $z_t\in \mathcal{D}(x,z)\subset \mathcal{D}_0(x)$.
\end{itemize}
Thus
\[
\{h_t(x):t\geqslant 0\}\;\cup\;\{z_t:t\geqslant 0\}\;\subset\;\mathcal{D}_0(x),
\]
and together with \eqref{eq:D0x-inclusion-calD} we obtain
\[
\mathcal{D}_0(x)=\{h_t(x):t\geqslant 0\}\;\cup\;\{z_t:t\geqslant 0\}.
\]

Since $h_t(u)-z_\infty=\rho^t(u-z_\infty)$, we have $h_t(u)\to z_\infty$ as $t\to\infty$ for $u\in\{x,p\}$, and $z_\infty$
is the unique accumulation point of $\mathcal{D}_0(x)$ in $(0,1)$. Hence
\[
\overline{\mathcal{D}_0(x)}=\mathcal{D}_0(x)\cup\{z_\infty\},
\]
and therefore
\[
\mathcal{D}(x)
=\overline{\mathcal{D}_0(x)}\cup\{0,1\}
=\{0,1\}\cup\{h_t(x):t\geqslant 0\}\cup\{z_t:t\geqslant 0\}\cup\{z_\infty\},
\]
which is the explicit expression claimed in Proposition~\ref{pro:reachable-set}.

Finally, $\mathcal{D}(x)$ is countable (a countable union plus finitely many points) and closed in $[0,1]$ by construction,
and \eqref{eq:Dxz-inclusion-calD} implies $\mathcal{D}(x,z)\subset \mathcal{D}_0(x)\subset \mathcal{D}(x)$ for every $z$.
Since $\mathcal{D}(x)\subset[0,1]$ is countable, closed, and contains $0,1$, Lemma~\ref{lem:partition-closed-countable}
applied with $K=\mathcal{D}(x)$ yields a partition of $[0,1]$ into left-semiclosed intervals or singletons having
$\mathcal{D}(x)$ as its endpoint set. This verifies the structural requirement in Lemma~\ref{lma:prop6}.
\end{proof}

\section{Proofs for Section 6 (Analytic Lagrangian relaxation and dual bound computation)}
\label{app:B}
Note that, in part (b), $F(p,z)$,  $G(p,z)$, $K_F(z)$ and $K_G(z)$ are as in Proposition~\textup{\ref{pro:edrwm}(b)}.

\begin{corollary}[Reward and work under Uniform initial belief]
\label{cor:edrwm-unif}
$F_{\nu_0}(z)$ and $G_{\nu_0}(z)$ admit the following closed-form
expressions.

\begin{enumerate}[label=(\alph*)]
\item If $0\leqslant z < p$, then
\[
F_{\nu_0}(z)
=
\frac{r}{1-\beta} - \frac{r}{2}\,z^2,
\quad
G_{\nu_0}(z)
=
\frac{1}{1-\beta} - z.
\]

\item If $p \leqslant z < z_\infty$, 
let $s=\tau(p, z)$, so
$
z_{s-1}\leqslant z<z_s,
$
and $t=\tau(0,z)$.  Define
\[
x_k(z)
\;\triangleq\;
h_k^{-1}(z)
\;=\;
z_\infty + \frac{z-z_\infty}{\rho^k},
\quad k=0,1,2,\ldots,
\]
\[
L_k(z)
\;\triangleq\;
\begin{cases}
x_{k-1}(z)-x_k(z), & k=1,\ldots,t-1,\\[2pt]
x_{t-1}(z),        & k=t,
\end{cases}
\]
and, for each $k\in\{1,\ldots,t\}$,
\[
B_k
=
\frac{1-(\beta\rho)^k}{1-\beta\rho},
\quad
A_k
=
\frac{1-\beta^k}{1-\beta}\,(1-z_\infty)
+
\frac{1-(\beta\rho)^k}{1-\beta\rho}\,z_\infty,
\]
so that $\Phi_k(x)=A_k-B_kx$.
Then,
\[
G_{\nu_0}(z)
=
K_G(z)\,S(z),
\quad
F_{\nu_0}(z)
=
r\,J(z)
\;+\;
K_F(z)\,S(z),
\]
where
\[
S(z)
\triangleq
\int_0^1 \beta^{\tau(x,z)}\,\mathrm{d}x
=
(1-z)
+
\sum_{k=1}^t \beta^k L_k(z),
\]
and hence
\[
S(z)=
\begin{cases}
1-z\;+\;\beta^t z_\infty
\;+\;\beta\,(z_\infty-z)\,
\dfrac{\,1-\rho-(1-\beta)\bigl(\tfrac{\beta}{\rho}\bigr)^{t-1}\,}{\rho-\beta},
& \beta\neq\rho,\\[10pt]
1-z\;+\;\rho^t z_\infty
\;+\;(z_\infty-z)\bigl[(1-\rho)(t-1)-\rho\bigr],
& \beta=\rho,
\end{cases}
\]
while
\[
\begin{aligned}
J(z)
&=\int_0^z \Phi_{\tau(x,z)}(x)\,\mathrm{d}x \\[2pt]
&=\sum_{k=1}^{t}
\left[
A_k\,L_k(z)
-\frac{B_k}{2}\bigl(x_{k-1}(z)^2-x_k(z)^2\bigr)
\right],
\end{aligned}
\]
where, for notational convenience, we set
$
x_t(z)\equiv 0,
L_t(z)\equiv x_{t-1}(z).
$

\item If $z_\infty\leqslant z < 1$, then
\[
G_{\nu_0}(z) = 1-z,
\]
and, with $\Phi_\infty(x)$ as in Section~\textup{\ref{s:ewrm}},
\[
F_{\nu_0}(z)
=
r\bigg[
\int_0^z \Phi_\infty(x)\,\mathrm{d}x
+ (1-z)\bigl(1+\beta\Phi_\infty(p)\bigr)
\bigg],
\]
which simplifies to the explicit affine--quadratic expression in $z$
obtained by direct integration.

\item If $z\geqslant 1$, then the policy is always passive and, independently of $z$,
\[
F_{\nu_0}(z)
=
r\bigg[
\frac{1-z_\infty}{1-\beta}
+
\frac{z_\infty-\tfrac12}{1-\beta\rho}
\bigg],
\quad
G_{\nu_0}(z)=0,
\]
\end{enumerate}
\end{corollary}
\begin{proof}
For each $z$ we set
$F_{\nu_0}(z)=\int_0^1 F(x,z)\,\mathrm{d}x$ and
$G_{\nu_0}(z)=\int_0^1 G(x,z)\,\mathrm{d}x$, and use the explicit
piecewise formulas of Proposition~\ref{pro:edrwm}.

\smallskip\noindent
\emph{(a), (c), (d):} In the regimes $0\leqslant z<p$, $z_\infty\leqslant z<1$,
and $z\geqslant 1$, the expressions for $F(x,z)$ and $G(x,z)$ are simple
piecewise affine (or constant) functions of $x$ on at most two
intervals ($x\leqslant z$ and $x>z$). Direct integration over $[0,1]$
produces the stated formulas after straightforward algebra.

\smallskip\noindent
\emph{(b) Case $p\leqslant z<z_\infty$.}
Fix $z$ with $p\leqslant z<z_\infty$. As in Proposition~\ref{pro:edrwm}(b), let
$s=s(z)\in\{1,2,\ldots\}$ be the unique integer such that
\[
z_{s-1}\leqslant z<z_s,
\]
so that
\[
F(p,z)
=
r\,\frac{\Phi_{s}(p)+\beta^{s}}{1-\beta^{s+1}},
\quad
G(p,z)
=
\frac{\beta^{s}}{1-\beta^{s+1}}.
\]
Define
\[
K_F(z)\;\triangleq\;r+\beta F(p,z),
\quad
K_G(z)\;\triangleq\;1+\beta G(p,z).
\]

For $x>z$, the $z$-threshold policy intervenes immediately at time $0$, so by
Proposition~\ref{pro:edrwm}(b),
\[
F(x,z)=K_F(z),
\quad
G(x,z)=K_G(z).
\]
For $x\leqslant z$, the belief follows the passive trajectory $h_t(x)$ until it
first exceeds $z$ at the finite time
\[
\tau(x,z)
\;\triangleq\;
\inf\{t\geqslant 1 : h_t(x)>z\},
\]
after which the policy intervenes and the belief jumps to $p$. Iterating the
dynamic programming equations along that passive spell yields (cf.\
Section~\ref{s:ewrm})
\[
G(x,z)
=
\beta^{\tau(x,z)}K_G(z),
\quad
F(x,z)
=
r\,\Phi_{\tau(x,z)}(x)
+\beta^{\tau(x,z)}K_F(z),
\]
where $\Phi_t(x)=\sum_{s=0}^{t-1}(1-h_s(x))\beta^s$.

Integrating over the uniform initial distribution $\nu_0$ on $[0,1]$ gives
\[
\begin{aligned}
G_{\nu_0}(z)
&=
\int_0^1 G(x,z)\,\mathrm{d}x\\
&=
\int_0^z \beta^{\tau(x,z)}K_G(z)\,\mathrm{d}x
\;+\;
\int_z^1 K_G(z)\,\mathrm{d}x\\
&=
K_G(z)\left[\int_0^z \beta^{\tau(x,z)}\,\mathrm{d}x + \int_z^1 1\,\mathrm{d}x\right]
=
K_G(z)\int_0^1 \beta^{\tau(x,z)}\,\mathrm{d}x,
\end{aligned}
\]
and
\[
\begin{aligned}
F_{\nu_0}(z)
&=
\int_0^1 F(x,z)\,\mathrm{d}x\\
&=
\int_0^z \Bigl[r\,\Phi_{\tau(x,z)}(x)
              +\beta^{\tau(x,z)}K_F(z)\Bigr]\,\mathrm{d}x
\;+\;
\int_z^1 K_F(z)\,\mathrm{d}x\\[3pt]
&=
r\int_0^z \Phi_{\tau(x,z)}(x)\,\mathrm{d}x
\;+\;
K_F(z)\Bigl[\int_0^z \beta^{\tau(x,z)}\,\mathrm{d}x + \int_z^1 1\,\mathrm{d}x\Bigr]\\[3pt]
&=
r\int_0^z \Phi_{\tau(x,z)}(x)\,\mathrm{d}x
\;+\;
K_F(z)\int_0^1 \beta^{\tau(x,z)}\,\mathrm{d}x.
\end{aligned}
\]

We now compute these integrals in closed form. Let
$t\triangleq\tau(0,z)\in\{2,3,\ldots\}$ denote the first passive
threshold--crossing time starting from $x=0$. Since $p\leqslant z<z_\infty$, one has
$t\geqslant 2$ and
\[
\tau(0,z)=t
\;\Longleftrightarrow\;
z_{t-2}\leqslant z<z_{t-1}.
\]
Define the inverse passive trajectory
\[
x_k(z)
\;\triangleq\;
h_k^{-1}(z)
\;=\;
z_\infty + \frac{z-z_\infty}{\rho^k},
\quad k=0,1,2,\ldots,
\]
so that \(x_0(z)=z\) and \(x_k(z)\) decreases with \(k\); for the finite \(t=\tau(0,z)\) used below, \(x_{t-1}(z)\ge0\) and \(x_t(z)<0\).
Then the set $[0,z]$ is partitioned into the disjoint intervals
\[
I_k(z)
=
\begin{cases}
\bigl(x_k(z),\,x_{k-1}(z)\bigr], & k=1,\ldots,t-1,\\[2pt]
\bigl[0,\,x_{t-1}(z)\bigr],      & k=t.
\end{cases}
\]
and on $I_k(z)$ we have $\tau(x,z)=k$. Writing
\[
L_k(z)
\;\triangleq\;
|I_k(z)|
=
\begin{cases}
x_{k-1}(z)-x_k(z), & k=1,\ldots,t-1,\\[2pt]
x_{t-1}(z),        & k=t,
\end{cases}
\]
a direct computation using the affine form of $x_k(z)$ shows that
\[
L_k(z)
=
\frac{(1-\rho)\bigl(z_\infty-z\bigr)}{\rho^k},
\quad k=1,\ldots,t-1,
\quad
L_t(z)=x_{t-1}(z)
=
z_\infty + \frac{z-z_\infty}{\rho^{t-1}}.
\]

Hence
\[
\int_0^1 \beta^{\tau(x,z)}\,\mathrm{d}x
=
\int_z^1 1\,\mathrm{d}x
+\sum_{k=1}^t \beta^k |I_k(z)|
=
(1-z)
+
\sum_{k=1}^t \beta^k L_k(z).
\]
For $\beta\neq\rho$ this sum can be written as
\[
(1-z)
+
(1-\rho)\bigl(z_\infty-z\bigr)\,
\frac{\tfrac{\beta}{\rho}\,\bigl(1-(\tfrac{\beta}{\rho})^{\,t-1}\bigr)}
     {1-\tfrac{\beta}{\rho}}
+
\beta^t\Bigl(
z_\infty + \frac{z-z_\infty}{\rho^{t-1}}
\Bigr),
\]
while in the critical case $\beta=\rho$ we use
\[
\sum_{k=1}^{t-1}\left(\frac{\beta}{\rho}\right)^k
=
\sum_{k=1}^{t-1}1
=t-1,
\]
which yields
\[
\int_0^1 \beta^{\tau(x,z)}\,\mathrm{d}x
=
(1-z)
+(1-\rho)\bigl(z_\infty-z\bigr)(t-1)
+\rho^t z_\infty + \rho\,(z-z_\infty).
\]

Similarly,
\[
\int_0^z \Phi_{\tau(x,z)}(x)\,\mathrm{d}x
=
\sum_{k=1}^t \int_{I_k(z)} \Phi_k(x)\,\mathrm{d}x,
\]
where, for each $k$,
\[
\Phi_k(x)
=
\sum_{s=0}^{k-1}\bigl(1-h_s(x)\bigr)\beta^s
=
A_k - B_k x
\]
with
\[
B_k
=
\frac{1-(\beta\rho)^k}{1-\beta\rho},
\quad
A_k
=
\frac{1-\beta^k}{1-\beta}\,(1-z_\infty)
+
\frac{1-(\beta\rho)^k}{1-\beta\rho}\,z_\infty.
\]
Note that $1-\beta\neq0$ and $1-\beta\rho\neq0$ for all $\beta\in(0,1)$
and $\rho\in(0,1)$, so these expressions remain valid also when
$\beta=\rho$. Hence, for $k=1,\ldots,t-1$,
\[
\int_{I_k(z)} \Phi_k(x)\,\mathrm{d}x
=
\int_{x_k(z)}^{x_{k-1}(z)} (A_k-B_kx)\,\mathrm{d}x
=
A_k L_k(z)
-\frac{B_k}{2}\bigl(x_{k-1}(z)^2 - x_k(z)^2\bigr),
\]
and, for $k=t$,
\[
\int_{I_t(z)} \Phi_t(x)\,\mathrm{d}x
=
\int_0^{x_{t-1}(z)} (A_t-B_t x)\,\mathrm{d}x
=
A_t x_{t-1}(z) - \frac{B_t}{2}x_{t-1}(z)^2.
\]
Therefore
\[
\int_0^z \Phi_{\tau(x,z)}(x)\,\mathrm{d}x
=
\sum_{k=1}^{t-1}
\left[
A_k L_k(z)
-\frac{B_k}{2}\bigl(x_{k-1}(z)^2 - x_k(z)^2\bigr)
\right]
+
A_t x_{t-1}(z) - \frac{B_t}{2}x_{t-1}(z)^2,
\]
which, together with the above expression for
$\int_0^1 \beta^{\tau(x,z)}\,\mathrm{d}x$, yields fully explicit formulas for
$F_{\nu_0}(z)$ and $G_{\nu_0}(z)$ in terms of the model parameters,
valid for all $\beta\in(0,1)$, including the critical case $\beta=\rho$.

\end{proof}

\section{Proofs for Section 7 (Dependence of the Whittle index on model parameters)}
\label{app:sec7}

\subsection{Proofs for Section 7.1 (Dependence on the lapse probability from adherence $p$)}
\label{app:sec71}

\begin{proof}[Proof of Proposition~\ref{pro:mp-monotone-p}]
Fix $x\in(0,1)$, $q\in(0,1)$, $\beta\in(0,1)$ and $r>0$, and view $m(x;p,q)$ as a function of $p\in(0,1-q)$.

\smallskip\noindent
\emph{Step 1: Constancy for $p\geqslant x$.}
If $p>x$, then $x<p$ and the first branch in Proposition~\ref{pro:dmpi-explicit} applies, giving
$m(x;p,q)=r x$, independent of $p$ on $(x,1-q)$. Also $m(x;x,q)=r x$ by continuity of the MP index in $x$
(Lemma~\ref{lma:pcli2}). Hence $p\mapsto m(x;p,q)$ is constant on $[x,1-q)$.

\smallskip\noindent
\emph{Step 2: Strict decrease for $0<p<x$.}
Assume $0<p<x$. There are two possibilities depending on whether $x$ lies on the last branch or on a middle branch.

\smallskip\noindent
\emph{(a) Last branch.}
The last branch applies when $z_\infty(p,q)\leqslant x$, i.e.,
\[
\frac{p}{p+q}\leqslant x
\quad\Longleftrightarrow\quad
p\leqslant p_\infty(x)\triangleq \frac{q\,x}{1-x}.
\]
(On $(0,1-q)$ this condition is vacuous when $x\geqslant 1-q$.) On this region,
\[
m(x;p,q)=\frac{r x}{1-\beta\rho}=\frac{r x}{1-\beta+\beta(p+q)},
\]
so $m$ is differentiable in $p$ and
\[
\frac{\partial}{\partial p}m(x;p,q)
=
-\,\frac{r x\,\beta}{\bigl(1-\beta+\beta(p+q)\bigr)^2}
<0.
\]
Hence $m(x;p,q)$ is decreasing in $p$ throughout the last-branch region.

\smallskip\noindent
\emph{(b) Middle branches.}
Now suppose $p_\infty(x)<p<x$, so $x<z_\infty(p,q)$ and therefore $x$ lies on some middle branch.
For each $t\ge1$ define the (possibly empty) branch set
\[
I_t(x)\triangleq \bigl\{\,p\in(0,1-q):\ z_{t-1}(p)\leqslant x<z_t(p)\,\bigr\},
\quad z_t(p)=h_t(p).
\]
On the interior $\mathrm{int}\,I_t(x)$ (where $t$ is fixed), Proposition~\ref{pro:dmpi-explicit} gives
\[
m(x;p,q)
=
\frac{r}{1-\beta}
\Bigl[(1-\beta^{t+1})\,x
      + \beta^{t+1}-\beta + \beta(1-\beta)\,\Phi_t(p)\Bigr],
\quad p\in \mathrm{int}\,I_t(x),
\]
so $m$ is differentiable in $p$ on $\mathrm{int}\,I_t(x)$ and
\[
\frac{\partial}{\partial p}m(x;p,q)
=
r\beta\,\Phi_t'(p).
\]
It remains to show $\Phi_t'(p)<0$. By \eqref{eq:Phi-t-p-sum},
\[
\Phi_t(p)=\sum_{s=0}^{t-1}(1-z_s)\beta^s,
\quad z_s=h_s(p),\quad h_{s+1}(p)=p+\rho(p)h_s(p),\ \rho(p)=1-p-q.
\]
Differentiating term-by-term yields
\[
\Phi_t'(p)=\sum_{s=0}^{t-1}-\,h_s'(p)\,\beta^s.
\]
Let $u_s(p)\triangleq h_s'(p)$. Differentiating the recursion and using $\rho'(p)=-1$ gives
\[
u_{s+1}(p)=1-h_s(p)+\rho(p)\,u_s(p),
\quad u_0(p)=1.
\]
For $p\in(0,1-q)$ we have $0<\rho(p)<1$, and $h_s(p)\nearrow z_\infty(p,q)=\frac{p}{p+q}<1$, hence $h_s(p)<1$ for all
finite $s$ and so $1-h_s(p)>0$. An induction on $s$ shows $u_s(p)>0$ for all $s\ge0$. Therefore each term
$-u_s(p)\beta^s$ is strictly negative, so $\Phi_t'(p)<0$ and thus
\[
\frac{\partial}{\partial p}m(x;p,q)<0
\quad\text{for all }p\in \mathrm{int}\,I_t(x).
\]
Hence $m(x;p,q)$ is decreasing on each nonempty interval $I_t(x)$.

\smallskip\noindent
\emph{Step 3: No upward jumps at middle-branch breakpoints.}
A switch between consecutive middle branches occurs at a breakpoint $\bar p$ satisfying $x=z_t(\bar p)$ for some $t\ge1$.
At such a breakpoint, the $t$- and $(t+1)$-branch formulas coincide. Indeed, write
\[
m_t(x;p,q)=\frac{r}{1-\beta}\Bigl[(1-\beta^{t+1})x+\beta^{t+1}-\beta+\beta(1-\beta)\Phi_t(p)\Bigr],
\]
and use the recursion $\Phi_{t+1}(p)=\Phi_t(p)+(1-z_t(p))\beta^t$. Substituting $x=z_t(p)$ yields
$m_t(z_t(p);p,q)=m_{t+1}(z_t(p);p,q)$. Thus $m(x;p,q)$ is continuous at every finite middle-branch breakpoint in $p$, and
since it is decreasing on each adjacent branch interval it cannot increase across the breakpoint.

\smallskip\noindent
\emph{Step 4: No upward jump at the middle/last boundary.}
If $x<1-q$, the boundary between the last branch and the middle branches is $p=p_\infty(x)=q x/(1-x)$, where
$z_\infty(p_\infty(x),q)=x$. On the last branch,
\[
m\bigl(x;p_\infty(x),q\bigr)=\frac{r x}{1-\beta+\beta\bigl(p_\infty(x)+q\bigr)}.
\]
Let $p_n\downarrow p_\infty(x)$ with $p_n>p_\infty(x)$. Then $x<z_\infty(p_n,q)$ so $x$ lies on a middle branch with
$t_n=\tau(p_n,x)\to\infty$. Using the middle-branch expression in Proposition~\ref{pro:dmpi-explicit} and the estimate
\[
\bigl|\Phi_{t_n}(p_n)-\Phi_\infty\bigl(p_\infty(x)\bigr)\bigr|
\le
\bigl|\Phi_{t_n}(p_n)-\Phi_\infty(p_n)\bigr|
+
\bigl|\Phi_\infty(p_n)-\Phi_\infty\bigl(p_\infty(x)\bigr)\bigr|,
\]
we have $\Phi_{t_n}(p_n)\to \Phi_\infty\bigl(p_\infty(x)\bigr)$ because
$|\Phi_{t}(p)-\Phi_\infty(p)|\leqslant \sum_{s=t}^\infty \beta^s=\beta^t/(1-\beta)$ and $\beta^{t_n}\to0$, while
$\Phi_\infty(\cdot)$ is continuous in $p$. Passing to the limit in the middle-branch formula yields
\[
\lim_{n\to\infty} m(x;p_n,q)=m\bigl(x;p_\infty(x),q\bigr),
\]
so $m(x;p,q)$ is continuous at $p=p_\infty(x)$ and cannot jump upward there.

\smallskip\noindent
\emph{Step 5: Conclusion.}
Steps 2--4 show that $m(x;p,q)$ is decreasing in $p$ on $(0,\min\{x,1-q\})$: it is decreasing on each open
branch interval and continuous at the (finite and limiting) breakpoints, so there are no upward jumps. Together with Step~1,
this implies that $p\mapsto m(x;p,q)$ is nonincreasing on $(0,1-q)$, decreasing on $(0,\min\{x,1-q\})$, and constant
on $[x,1-q)$.
\end{proof}

\subsection{Proofs for Section 7.2 (Dependence on the spontaneous recovery probability $q$)}
\label{app:sec72}

\begin{proof}[Proof of Proposition~\ref{pro:mp-monotone-q}]
Fix $p\in(0,1)$, $\beta\in(0,1)$, $r>0$ and $x\in(0,1)$, and view $m(x;p,q)$ as a function of $q\in(0,1-p)$.
Let $\rho=1-p-q$. Recall the passive belief update
\[
h(u)=p+\rho u,\quad \rho=1-p-q,
\]
and define the iterates
\[
h_0(p,q)=p,\quad h_{s+1}(p,q)=p+\rho\,h_s(p,q),\quad s\ge0,
\]
together with the limiting point
\[
z_\infty(p,q)=\frac{p}{1-\rho}=\frac{p}{p+q}.
\]
From Proposition~\ref{pro:dmpi-explicit} we have the piecewise representation
\[
m(x;p,q) =
\begin{cases}
r x, & 0\leqslant x < p,\\[4pt]
\dfrac{r}{1-\beta}
\begin{aligned}[t]
\Bigl[&(1-\beta^{t+1})\,x
      + \beta^{t+1}-\beta \\
      &\qquad + \beta(1-\beta)\,\Phi_t(p,q)\Bigr],
\end{aligned}
& h_{t-1}(p,q)\leqslant x < h_t(p,q)\ \text{for some }t\geqslant 1,\\[8pt]
\dfrac{r\,x}{1-\beta\rho},
        & z_\infty(p,q)\leqslant x \leqslant 1,
\end{cases}
\]
where
\[
\Phi_t(p,q)
=
\sum_{s=0}^{t-1} \bigl(1-h_s(p,q)\bigr)\,\beta^{s}.
\]

\smallskip\noindent
\emph{(i) First branch $0\leqslant x<p$.}
If $x<p$, the first line gives $m(x;p,q)=r x$, independent of $q$. This proves (i).

\smallskip\noindent
Assume henceforth that $x\geqslant p$.

\smallskip\noindent
\emph{(ii.a) First middle branch ($t=1$).}
The first middle branch applies whenever $p\leqslant x<h_1(p,q)$.
On this branch,
\[
m(x;p,q)
=
\frac{r}{1-\beta}
\Bigl[(1-\beta^{2})\,x
      + \beta^{2}-\beta + \beta(1-\beta)\,\Phi_1(p,q)\Bigr],
\]
and
\[
\Phi_1(p,q)=\sum_{s=0}^{0} (1-h_s(p,q))\beta^s = 1-h_0(p,q)=1-p,
\]
which does not depend on $q$. Hence $m(x;p,q)$ is independent of $q$ on this branch, proving (ii)(a).

\smallskip\noindent
\emph{(ii.b) Higher middle branches ($t\geqslant 2$).}
Fix $t\geqslant 2$ and consider a point $(p,q)$ such that $h_{t-1}(p,q)\leqslant x<h_t(p,q)$.
On this branch,
\[
m(x;p,q)
=
\frac{r}{1-\beta}
\Bigl[(1-\beta^{t+1})\,x
      + \beta^{t+1}-\beta + \beta(1-\beta)\,\Phi_t(p,q)\Bigr].
\]
Thus, on the interior of such a branch region, differentiating with respect to $q$ yields
\[
\frac{\partial}{\partial q} m(x;p,q)
=
r\beta\,\frac{\partial}{\partial q}\Phi_t(p,q).
\]
It therefore suffices to show $\partial \Phi_t(p,q)/\partial q>0$ for $t\geqslant 2$.

Write $u_s(q)\triangleq \frac{\partial}{\partial q}h_s(p,q)$ for fixed $p$.
Differentiating the recursion $h_{s+1}(p,q)=p+\rho h_s(p,q)$ with respect to $q$ and using $\rho'(q)=-1$ gives
\[
u_{s+1}(q)
=
-\,h_s(p,q)+\rho\,u_s(q),
\quad u_0(q)=\frac{\partial}{\partial q}h_0(p,q)=0.
\]
Since $0<\rho<1$ for $q\in(0,1-p)$ and $0<h_s(p,q)<1$ for all $s$ (because $h_s(p,q)\nearrow z_\infty(p,q)<1$),
we have
\[
u_1(q)=-h_0(p,q)=-p<0,
\]
and if $u_s(q)<0$ then $u_{s+1}(q)=-h_s(p,q)+\rho u_s(q)<0$. Hence
\[
u_0(q)=0,\quad u_s(q)<0\ \ \text{for all }s\ge1.
\]
Now
\[
\frac{\partial}{\partial q}\Phi_t(p,q)
=
\sum_{s=0}^{t-1} -\,\frac{\partial h_s(p,q)}{\partial q}\,\beta^s
=
-\sum_{s=0}^{t-1} u_s(q)\,\beta^s
=
-\sum_{s=1}^{t-1} u_s(q)\,\beta^s,
\]
which is positive for $t\geqslant 2$ because the sum is nonempty and each term satisfies $-u_s(q)\beta^s>0$.
Therefore $\frac{\partial}{\partial q}\Phi_t(p,q)>0$ and hence
\[
\frac{\partial}{\partial q}m(x;p,q)=r\beta\,\frac{\partial}{\partial q}\Phi_t(p,q)>0
\]
on every higher middle branch ($t\geqslant 2$), proving (ii)(b).

\smallskip\noindent
\emph{(iii) Last branch.}
The last branch applies when $z_\infty(p,q)\leqslant x$, i.e.,
\[
\frac{p}{p+q}\leqslant x.
\]
For fixed $p$, the map $q\mapsto z_\infty(p,q)=p/(p+q)$ is decreasing on $(0,1-p)$.
For each $x\in(p,1)$ there is a unique $q_\infty(x)\in(0,1-p)$ such that $z_\infty(p,q_\infty(x))=x$, obtained by solving
$x=p/(p+q)$ for $q$, which gives
\[
q_\infty(x)=\frac{p(1-x)}{x}.
\]
Thus $z_\infty(p,q)\leqslant x$ is equivalent to $q\geqslant q_\infty(x)$, so the last branch applies exactly on
$J_\infty(x)=[q_\infty(x),\,1-p)$ (for $x>p$; it is empty for $x\leqslant p$).

On the last branch,
\[
m(x;p,q)=\frac{r x}{1-\beta\rho}=\frac{r x}{1-\beta+\beta(p+q)}.
\]
Let $D(q)=1-\beta+\beta(p+q)$, so $D'(q)=\beta>0$. Then
\[
\frac{\partial}{\partial q}m(x;p,q)
=
r x\,\frac{\partial}{\partial q}\left(\frac{1}{D(q)}\right)
=
-\,r x\,\frac{D'(q)}{D(q)^2}
=
-\,\frac{r x\beta}{D(q)^2}<0.
\]
Hence $m(x;p,q)$ is decreasing in $q$ on the last branch, proving (iii).

\smallskip
Combining (i), (ii) and (iii) yields the stated behavior and completes the proof.
\end{proof}

\rev{
\section{Proofs for Section~\ref{s:iatac} (Long-run average criterion)}
\label{app:average-metrics}

This appendix checks the model-specific ingredients used in Section~\ref{s:iatac} to apply the long-run average PCL framework of \citet[Secs.~3--4]{nmwp0726}. We do not repeat the criterion-agnostic PCL verification theorem, the generalized-inverse threshold argument, or the marginal sign lemma; those are \citet[Theorem~3.6, Lemmas~3.3--3.4, and Theorem~4.5]{nmwp0726}.

\begin{proof}[Proof of Proposition~\ref{pro:avg-acoe}]
On the invariant core \(\mathsf X^{\mathrm{inv}}=[0,z_\infty]\), the discounted \(\lambda\)-price Bellman operator is
\[
(T_{\beta,\lambda}V)(x)
=
\max\{r-\lambda+\beta V(p),\; r(1-x)+\beta V(h(x))\},
\qquad x\in\mathsf X^{\mathrm{inv}}.
\]
Let \(L=r/(1-\rho)\). If \(V\) is \(L\)-Lipschitz on \(\mathsf X^{\mathrm{inv}}\), the active branch is constant and the passive branch has Lipschitz modulus at most
\[
r+\beta\rho L\le r+\rho L=L.
\]
Hence \(T_{\beta,\lambda}\) leaves the same Lipschitz class invariant for all \(0<\beta<1\). The one-period net reward is bounded on the compact core, and the transition kernel is weakly continuous because the next state is either \(p\) or \(h(x)=p+\rho x\). Thus the compact-state relative-compactness condition in \citet[Lemma~4.2]{nmwp0726} holds. The standard vanishing-discount average-optimality argument summarized in \citet[Sec.~4.1]{nmwp0726} gives a state-independent gain and a continuous bias satisfying \eqref{eq:avg-acoe-adherence}, together with the optimal-action characterization in \citet[Assumption~4.1]{nmwp0726}.
\end{proof}

\begin{proof}[Proof of Proposition~\ref{pro:avg-rwm}]
The formulas are the Abelian limits of Proposition~\ref{pro:edrwm}; equivalently, they follow directly from the regenerative structure of threshold trajectories on \(\mathsf X^{\mathrm{inv}}\).

If \(z<p\), then after at most one transient initial passive step the threshold policy intervenes in every period. A finite transient does not affect long-run averages, so \(F^{\mathrm{avg}}(z)=r\) and \(G^{\mathrm{avg}}(z)=1\).

If \(p\le z<z_\infty\), let \(s=s(z)\ge1\) satisfy \(z_{s-1}\le z<z_s\). Starting from \(p\), the threshold trajectory has the deterministic cycle
\[
p=z_0,z_1,\ldots,z_{s-1},\quad \text{activation},\quad p,
\]
of length \(s+1\). The cycle contains one activation and total reward \(r\{1+\bar\Phi_s(p)\}\), where the term \(1\) is the intervention-period adherence reward. Hence
\[
G^{\mathrm{avg}}(z)=\frac1{s+1},
\qquad
F^{\mathrm{avg}}(z)=r\,\frac{1+\bar\Phi_s(p)}{s+1}.
\]

At \(z=z_\infty\), passivity never crosses the threshold on the core. Thus the policy is all-passive on \(\mathsf X^{\mathrm{inv}}\), \(G^{\mathrm{avg}}(z_\infty)=0\), and the passive belief converges to \(z_\infty\), giving \(F^{\mathrm{avg}}(z_\infty)=r(1-z_\infty)\).
\end{proof}

\begin{proof}[Proof of Proposition~\ref{pro:avg-mrwm}]
We take Abelian bias limits of the discounted marginal formulas in Proposition~\ref{pro:dmrwm}, restricted to \(x\in\mathsf X^{\mathrm{inv}}\) and \(z\in\mathsf Z^{\mathrm{inv}}\). This is the model-specific version of the bias-Q marginal passage described in \citet[Sec.~4.2]{nmwp0726}.

If \(z<p\), Proposition~\ref{pro:dmrwm}\textup{(a)} gives \(f_\beta(x,z)=rx\) and \(g_\beta(x,z)=1\) for every \(\beta\), proving part \textup{(a)}.

Let \(p\le z<z_\infty\), and let \(s=s(z)=\tau(p,z)\), so \(z_{s-1}\le z<z_s\). Since
\[
G_\beta(p,z)=\frac{\beta^s}{1-\beta^{s+1}},
\qquad
K_G^\beta(z)\triangleq1+\beta G_\beta(p,z)=\frac1{1-\beta^{s+1}},
\]
Proposition~\ref{pro:dmrwm}\textup{(b)} gives, for \(x\le z\) and \(t=\tau(x,z)\),
\[
g_\beta(x,z)=\frac{1-\beta^t}{1-\beta^{s+1}}
\longrightarrow \frac{t}{s+1}.
\]
Moreover,
\[
K_F^\beta(z)\triangleq r+\beta F_\beta(p,z)
=
\frac{r\{1+\beta\Phi_{s,\beta}(p)\}}{1-\beta^{s+1}},
\]
where \(\Phi_{s,\beta}\) denotes the discounted passive sum. Since \(\Phi_{k,\beta}(u)\to\bar\Phi_k(u)\) for each fixed finite \(k\),
\[
(1-\beta^t)K_F^\beta(z)-r\Phi_{t,\beta}(x)
\longrightarrow
r\left[
\frac{t}{s+1}\bigl(1+\bar\Phi_s(p)\bigr)-\bar\Phi_t(x)
\right].
\]
This proves part \textup{(b)(i)}. If \(x>z\), Proposition~\ref{pro:dmrwm}\textup{(b)} gives
\[
g_\beta(x,z)=(1-\beta)K_G^\beta(z),
\qquad
f_\beta(x,z)=r(x-1)+(1-\beta)K_F^\beta(z),
\]
and hence
\[
g^{\mathrm{avg}}(x,z)=\frac1{s+1},
\qquad
f^{\mathrm{avg}}(x,z)=r(x-1)+r\frac{1+\bar\Phi_s(p)}{s+1}
=r(x-1)+F^{\mathrm{avg}}(z),
\]
which proves part \textup{(b)(ii)}.

It remains to consider the endpoint threshold \(z=z_\infty\). For \(x\in\mathsf X^{\mathrm{inv}}\), the off-core condition \(x>z\) cannot occur. Therefore the nonexceptional branch of Proposition~\ref{pro:dmrwm}\textup{(c)} gives
\[
g_\beta(x,z_\infty)=1,
\qquad
f_\beta(x,z_\infty)=\frac{rx}{1-\beta\rho},
\]
so
\[
g^{\mathrm{avg}}(x,z_\infty)=1,
\qquad
f^{\mathrm{avg}}(x,z_\infty)=\frac{rx}{1-\rho}.
\]
The displayed formulas also show \(g^{\mathrm{avg}}(x,z)>0\) for all \(x\in\mathsf X^{\mathrm{inv}}\) and \(z\in\mathsf Z^{\mathrm{inv}}\), proving \eqref{eq:avg-apcli1-core}.
\end{proof}

\begin{proof}[Proof of Proposition~\ref{pro:avg-mpi-explicit}]
By \eqref{eq:avg-apcli1-core}, the diagonal ratio defining \(m^{\mathrm{avg}}\) is well defined on \(\mathsf X^{\mathrm{inv}}\).

If \(0\le x<p\), Proposition~\ref{pro:avg-mrwm}\textup{(a)} with \(z=x\) gives \(m^{\mathrm{avg}}(x)=rx\). If \(p\le x<z_\infty\), let \(t=\tau(p,x)\), equivalently \(z_{t-1}\le x<z_t\). With \(z=x\), we have \(\tau(x,x)=1\), since \(h(x)>x\) for \(x<z_\infty\). Proposition~\ref{pro:avg-mrwm}\textup{(b)(i)} gives
\[
g^{\mathrm{avg}}(x,x)=\frac1{t+1},
\qquad
f^{\mathrm{avg}}(x,x)
=
r\left[\frac{1+\bar\Phi_t(p)}{t+1}-(1-x)\right],
\]
and therefore
\[
m^{\mathrm{avg}}(x)=r\bigl[(t+1)x+\bar\Phi_t(p)-t\bigr].
\]
At \(x=z_\infty\), Proposition~\ref{pro:avg-mrwm}\textup{(c)} gives
\[
g^{\mathrm{avg}}(z_\infty,z_\infty)=1,
\qquad
f^{\mathrm{avg}}(z_\infty,z_\infty)=\frac{r z_\infty}{1-\rho},
\]
which yields the endpoint value.

The displayed formula is affine on each interval \((0,p)\) and \((z_{t-1},z_t)\), \(t\ge1\), with respective slopes \(r\) and \(r(t+1)\), all strictly positive. At \(p=z_0\), the first branch gives \(rp\), and the \(t=1\) middle branch gives
\[
r\{2p+\bar\Phi_1(p)-1\}=r\{2p+(1-p)-1\}=rp.
\]
At an internal breakpoint \(z_t\), the \(t\) and \(t+1\) middle branches agree because \(\bar\Phi_{t+1}(p)=\bar\Phi_t(p)+(1-z_t)\). Finally, as \(t\to\infty\), \(z_t\uparrow z_\infty\) and the middle-branch values converge to \(r z_\infty/(1-\rho)\), which equals the endpoint value. Hence \(m^{\mathrm{avg}}\) is continuous and strictly increasing on \(\mathsf X^{\mathrm{inv}}\).
\end{proof}

\begin{proof}[Proof of Proposition~\ref{pro:avg-pcl-verification}]
The positivity condition \textup{(APCLI1)} is exactly \eqref{eq:avg-apcli1-core}, and \textup{(APCLI2)} follows from Proposition~\ref{pro:avg-mpi-explicit}. It remains to verify the two PCL identities.

For \textup{(APCLI3a)}, the threshold gain metrics are right-continuous step functions of the threshold. Put \(\bar\Phi_0(p)=0\) and, for \(s\ge0\), define
\[
F_s\triangleq r\frac{1+\bar\Phi_s(p)}{s+1},
\qquad
G_s\triangleq\frac1{s+1},
\]
where \(s=0\) denotes the all-active interval \((-\infty,p)\), so \(F_0=r\) and \(G_0=1\). On \([z_{s-1},z_s)\), \(s\ge1\), the threshold metrics are \(F_s,G_s\). At the atom \(z_t\), \(t\ge0\), right-continuity gives the jump from \((F_t,G_t)\) to \((F_{t+1},G_{t+1})\). Hence
\[
\Delta G_t=G_{t+1}-G_t=-\frac1{(t+1)(t+2)},
\]
and, using \(\bar\Phi_{t+1}(p)=\bar\Phi_t(p)+(1-z_t)\),
\[
\Delta F_t
=F_{t+1}-F_t
=-\frac{r\{(t+1)z_t+\bar\Phi_t(p)-t\}}{(t+1)(t+2)}
=m^{\mathrm{avg}}(z_t)\Delta G_t.
\]
The jump tails are summable and there is no additional atom at \(z_\infty\). Summing these atom identities over the jump points in \((z_1,z_2]\) proves
\[
F^{\mathrm{avg}}(z_2)-F^{\mathrm{avg}}(z_1)
=
\int_{(z_1,z_2]}m^{\mathrm{avg}}(y)\,\mathrm{d}G^{\mathrm{avg}}(y),
\qquad z_1<z_2\le z_\infty,
\]
with endpoint-constant continuation on \((-\infty,p)\). This is \textup{(APCLI3a)}.

For \textup{(APCLI3b)}, fix \(x\in\mathsf X^{\mathrm{inv}}\) and put
\[
D_x(z)=f^{\mathrm{avg}}(x,z)-m^{\mathrm{avg}}(z)g^{\mathrm{avg}}(x,z),
\qquad z\in\mathsf X^{\mathrm{inv}}.
\]
Direct substitution of Propositions~\ref{pro:avg-mrwm} and~\ref{pro:avg-mpi-explicit} gives, for \(0\le z<z_\infty\),
\[
D_x(z)=
\begin{cases}
r(x-z), & z<x,\\[0.7ex]
r\bigl[t(1-z)-\bar\Phi_t(x)\bigr], & x\le z<z_\infty,
\quad t=\tau(x,z),
\end{cases}
\]
and at the endpoint
\[
D_x(z_\infty)=\frac{r(x-z_\infty)}{1-\rho}.
\]
The displayed formulas are continuous at the threshold-orbit breakpoints. Indeed, if \(c=h_j(x)\), the crossing index changes from \(j\) to \(j+1\), and the identity \(\bar\Phi_{j+1}(x)=\bar\Phi_j(x)+(1-c)\) makes the two adjacent values coincide; the breakpoints \(z_t\) are handled similarly using the continuity of \(m^{\mathrm{avg}}\), and the endpoint value is the limit as \(z\uparrow z_\infty\).

On every open interval on which the indices \(s=s(z)\) and \(t=\tau(x,z)\) are fixed, \(m^{\mathrm{avg}}\) is affine and \(g^{\mathrm{avg}}(x,\cdot)\) is constant. If \(z<x\), then \(D_x'(z)=-r\), while \(g^{\mathrm{avg}}(x,z)(m^{\mathrm{avg}})'(z)=r\), since either \(z<p\) or \(g^{\mathrm{avg}}(x,z)=1/(s+1)\) and \((m^{\mathrm{avg}})'(z)=r(s+1)\). If \(x\le z<z_\infty\), then \(D_x'(z)=-rt\). For \(z<p\), necessarily \(t=1\), \(g^{\mathrm{avg}}(x,z)=1\), and \((m^{\mathrm{avg}})'(z)=r\); for \(p\le z<z_\infty\), \(g^{\mathrm{avg}}(x,z)=t/(s+1)\) and \((m^{\mathrm{avg}})'(z)=r(s+1)\). Hence in both cases \(g^{\mathrm{avg}}(x,z)(m^{\mathrm{avg}})'(z)=rt\). Thus, in all continuity intervals,
\[
D_x'(z)=-g^{\mathrm{avg}}(x,z)(m^{\mathrm{avg}})'(z).
\]
Since \(m^{\mathrm{avg}}\) is continuous, its Lebesgue--Stieltjes measure has no atoms; summing the last differential identity over the finitely many subintervals below any \(z_t\), and then passing to the summable limit as needed near \(z_\infty\), yields
\[
D_x(z)=
\begin{cases}
\displaystyle
\int_{(z,x)}g^{\mathrm{avg}}(x,y)\,\mathrm{d}m^{\mathrm{avg}}(y), & z<x,\\[2ex]
\displaystyle
-\int_{[x,z]}g^{\mathrm{avg}}(x,y)\,\mathrm{d}m^{\mathrm{avg}}(y), & x\le z.
\end{cases}
\]
This is \textup{(APCLI3b)}. The sign identity \eqref{eq:avg-marginal-sign} is then \citet[Lemma~3.4]{nmwp0726}, applied with empty exceptional set.
\end{proof}

\begin{proof}[Proof of Theorem~\ref{the:avg-indexability}]
The discounted project is PCL-indexable for every \(\beta\in(0,1)\) by Theorem~\ref{the:adherence-indexability}. Proposition~\ref{pro:avg-acoe} verifies the average optimal-action structure required in \citet[Assumption~4.1]{nmwp0726}. Propositions~\ref{pro:avg-rwm}--\ref{pro:avg-pcl-verification} verify the gain, bias-marginal, positivity, monotonicity, regularity, and PCL identities on the invariant core, with no exceptional core states. Therefore \citet[Theorem~4.5]{nmwp0726} applies and gives threshold-indexability under the average criterion, identifies the Whittle index with \(m^{\mathrm{avg}}\), and identifies the optimal threshold maps as the generalized inverses of \(m^{\mathrm{avg}}\).
\end{proof}
}

\rev{
\section{Additional numerical results for Section~\ref{s:numerics}}
\label{supp:numerics}

This appendix collects supplementary numerical results for the three experiment
families introduced in Section~\ref{s:numerics}. Recall that all experiments
use \(\beta=0.99\), truncation horizon \(T_\pi=700\), and \(300\) Monte Carlo
runs per instance. The material below complements the main-text tables by
providing more disaggregated summaries and by documenting the instances on which
the myopic policy performs particularly poorly.

\subsection{Detailed summaries by experiment family}
\label{supp:numerics-summaries}

Tables~\ref{tab:S-gaps-exp1}--\ref{tab:S-gaps-exp3} report the full
distributions of relative Lagrangian gaps \(\gamma^\pi\) for the six policies
considered in the paper, separately for the three experiment families:
Experiment~1 (two-type stress), Experiment~2 (three-type stress), and
Experiment~3 (jittered finite-mixture heterogeneity). These tables complement
Table~\ref{tab:gaps-three-experiments} in the main text by showing the
dispersion of the gaps within each experiment family.

For compactness in the ratio tables, we write
\(\mathrm{Wh}\), \(\mathrm{my}\), \(\mathrm{Lag}\), and \(\mathrm{LagF}\)
for the Whittle, myopic, Lagrangian, and forced-capacity Lagrangian
index policies, respectively.

Likewise, Tables~\ref{tab:S-ratios-exp1}--\ref{tab:S-ratios-exp3} report the
corresponding distributions of reward ratios relative to the Whittle index
policy. Of
particular interest is the ratio
\(\bar V^{\mathrm{Wh}}/\bar V^{\mathrm{my}}\), which quantifies the practical improvement of Whittle's policy over
the myopic rule, and the ratio
\(\bar V^{\mathrm{LagF}}/\bar V^{\mathrm{Wh}}\), which confirms that the
forced-capacity Lagrangian index benchmark is nearly indistinguishable from Whittle's policy across
all three designs.

\begin{table}[!htbp]
\centering
\caption{Summary of relative Lagrangian gaps \(\gamma^\pi\) (in \%) for Experiment~1: two-type stress regime.}
\label{tab:S-gaps-exp1}
\begin{tabular}{lccccccc}
\hline
Policy & mean & std & min & \(q_{0.25}\) & median & \(q_{0.75}\) & max \\
\hline
Whittle index               & 0.40 & 0.25 & 0.11 & 0.21 & 0.35 & 0.50 & 1.34 \\
Myopic index           & 10.10 & 9.76 & 0.21 & 2.80 & 6.20 & 15.13 & 44.33 \\
Round-robin            & 3.59 & 2.66 & 0.36 & 1.73 & 3.09 & 4.77 & 13.94 \\
Lagrangian             & 1.95 & 1.23 & 0.34 & 1.05 & 1.83 & 2.37 & 8.67 \\
Forced-capacity Lag.   & 0.39 & 0.24 & 0.11 & 0.21 & 0.33 & 0.48 & 1.34 \\
Passive                & 61.64 & 13.17 & 32.47 & 52.72 & 62.89 & 71.54 & 86.03 \\
\hline
\end{tabular}
\end{table}

\begin{table}[!htbp]
\centering
\caption{Summary of relative Lagrangian gaps \(\gamma^\pi\) (in \%) for Experiment~2: three-type stress regime.}
\label{tab:S-gaps-exp2}
\begin{tabular}{lccccccc}
\hline
Policy & mean & std & min & \(q_{0.25}\) & median & \(q_{0.75}\) & max \\
\hline
Whittle index               & 0.43 & 0.20 & 0.20 & 0.30 & 0.35 & 0.53 & 1.09 \\
Myopic index           & 12.24 & 8.00 & 1.23 & 6.69 & 10.30 & 16.79 & 40.99 \\
Round-robin            & 5.23 & 2.13 & 1.46 & 3.70 & 5.57 & 6.58 & 10.63 \\
Lagrangian             & 1.75 & 0.79 & 0.30 & 1.22 & 1.74 & 2.24 & 4.91 \\
Forced-capacity Lag.   & 0.40 & 0.18 & 0.18 & 0.27 & 0.33 & 0.47 & 1.03 \\
Passive                & 63.80 & 12.22 & 37.91 & 53.83 & 64.32 & 72.91 & 84.96 \\
\hline
\end{tabular}
\end{table}

\begin{table}[!htbp]
\centering
\caption{Summary of relative Lagrangian gaps \(\gamma^\pi\) (in \%) for Experiment~3: jittered finite-mixture population.}
\label{tab:S-gaps-exp3}
\begin{tabular}{lccccccc}
\hline
Policy & mean & std & min & \(q_{0.25}\) & median & \(q_{0.75}\) & max \\
\hline
Whittle index               & 0.39 & 0.08 & 0.27 & 0.32 & 0.36 & 0.45 & 0.61 \\
Myopic index           & 9.05 & 3.32 & 2.96 & 5.40 & 9.23 & 12.22 & 14.50 \\
Round-robin            & 4.94 & 1.21 & 2.77 & 4.07 & 4.68 & 5.60 & 7.91 \\
Lagrangian             & 1.28 & 0.32 & 0.69 & 1.06 & 1.24 & 1.52 & 2.04 \\
Forced-capacity Lag.   & 0.36 & 0.08 & 0.25 & 0.30 & 0.34 & 0.41 & 0.59 \\
Passive                & 58.52 & 7.00 & 44.59 & 53.14 & 58.91 & 63.18 & 71.15 \\
\hline
\end{tabular}
\end{table}

\begin{table}[!htbp]
\centering
\caption{Summary of reward ratios relative to Whittle's policy for Experiment~1: two-type stress regime.}
\label{tab:S-ratios-exp1}
\begin{tabular}{lccccccc}
\hline
Metric & mean & std & min & \(q_{0.25}\) & median & \(q_{0.75}\) & max \\
\hline
\(\bar V^{\mathrm{Wh}}/\bar V^{\mathrm{my}}\) & 1.124 & 0.153 & 1.000 & 1.023 & 1.064 & 1.172 & 1.787 \\
\(\bar V^{\mathrm{Lag}}/\bar V^{\mathrm{Wh}}\)    & 0.984 & 0.013 & 0.915 & 0.981 & 0.987 & 0.993 & 1.000 \\
\(\bar V^{\mathrm{LagF}}/\bar V^{\mathrm{Wh}}\)   & 1.000 & 0.000 & 0.999 & 1.000 & 1.000 & 1.000 & 1.001 \\
\hline
\end{tabular}
\end{table}

\begin{table}[!htbp]
\centering
\caption{Summary of reward ratios relative to Whittle's policy for Experiment~2: three-type stress regime.}
\label{tab:S-ratios-exp2}
\begin{tabular}{lccccccc}
\hline
Metric & mean & std & min & \(q_{0.25}\) & median & \(q_{0.75}\) & max \\
\hline
\(\bar V^{\mathrm{Wh}}/\bar V^{\mathrm{my}}\) & 1.146 & 0.125 & 1.009 & 1.067 & 1.108 & 1.198 & 1.690 \\
\(\bar V^{\mathrm{Lag}}/\bar V^{\mathrm{Wh}}\)    & 0.987 & 0.008 & 0.953 & 0.982 & 0.987 & 0.993 & 1.000 \\
\(\bar V^{\mathrm{LagF}}/\bar V^{\mathrm{Wh}}\)   & 1.000 & 0.000 & 1.000 & 1.000 & 1.000 & 1.001 & 1.001 \\
\hline
\end{tabular}
\end{table}

\begin{table}[!htbp]
\centering
\caption{Summary of reward ratios relative to Whittle's policy for Experiment~3: jittered finite-mixture population.}
\label{tab:S-ratios-exp3}
\begin{tabular}{lccccccc}
\hline
Metric & mean & std & min & \(q_{0.25}\) & median & \(q_{0.75}\) & max \\
\hline
\(\bar V^{\mathrm{Wh}}/\bar V^{\mathrm{my}}\) & 1.097 & 0.039 & 1.027 & 1.054 & 1.097 & 1.135 & 1.166 \\
\(\bar V^{\mathrm{Lag}}/\bar V^{\mathrm{Wh}}\)    & 0.991 & 0.003 & 0.984 & 0.989 & 0.991 & 0.993 & 0.997 \\
\(\bar V^{\mathrm{LagF}}/\bar V^{\mathrm{Wh}}\)   & 1.000 & 0.000 & 1.000 & 1.000 & 1.000 & 1.000 & 1.000 \\
\hline
\end{tabular}
\end{table}

\subsection{Worst instances for the myopic policy in the two-type stress regime}
\label{supp:worst-myopic}

Table~\ref{tab:S-worst-myopic} lists the ten simulation instances in
Experiment~1 with the largest ratio between the myopic and Whittle relative
Lagrangian gaps, \(\gamma^{\mathrm{my}}/\gamma^{\mathrm{Wh}}\), and also
reports the corresponding gaps of the round-robin, Lagrangian,
forced-capacity Lagrangian, and passive baselines. Because Experiment~1 varies
both reward weights and population size, the table includes \((r^A,r^B)\) and \(N\). The ten worst cases are concentrated in one highly adverse structural
regime: a balanced population with a substantially more lapse-prone type \(B\),
very low spontaneous recovery probabilities, and tight capacity. The repeated
appearance of this regime across several values of \(N\) and reward profiles
indicates that the poor performance of the myopic rule is not a finite-size
artifact.

\begin{table}[!htbp]
\centering
\caption{Ten instances in Experiment~1 with the largest ratio between the myopic and Whittle relative Lagrangian gaps.}
\label{tab:S-worst-myopic}
\scriptsize
\setlength{\tabcolsep}{2.2pt}
\renewcommand{\arraystretch}{1.05}
\begin{tabular}{ccccccccc cccccc c}
\hline
\(p^A\) & \(q^A\) & \(p^B\) & \(q^B\) & \(r^A\) & \(r^B\) &
\(\mathrm{prop}_A\) & \(N\) & \(M/N\) &
\(\gamma^{\mathrm{Wh}}\) & \(\gamma^{\mathrm{my}}\) & \(\gamma^{\mathrm{rr}}\) &
\(\gamma^{\mathrm{Lag}}\) & \(\gamma^{\mathrm{LagF}}\) & \(\gamma^{\mathrm{pas}}\) &
\(\gamma^{\mathrm{my}}/\gamma^{\mathrm{Wh}}\) \\
\hline
0.05 & 0.01 & 0.35 & 0.01 & 1 & 1 & 0.5 & 200  & 0.1 & 0.20 & 32.73 & 1.25 & 1.53 & 0.20 & 77.05 & 161.10 \\
0.05 & 0.01 & 0.35 & 0.01 & 2 & 2 & 0.5 & 200  & 0.1 & 0.21 & 32.73 & 1.26 & 1.51 & 0.20 & 77.08 & 159.45 \\
0.05 & 0.01 & 0.35 & 0.01 & 1 & 1 & 0.5 & 100  & 0.1 & 0.21 & 32.71 & 1.26 & 1.55 & 0.20 & 77.06 & 158.82 \\
0.05 & 0.01 & 0.35 & 0.01 & 1 & 1 & 0.5 & 700  & 0.1 & 0.21 & 32.75 & 1.26 & 1.48 & 0.21 & 77.07 & 158.16 \\
0.05 & 0.01 & 0.35 & 0.01 & 1 & 1 & 0.5 & 400  & 0.1 & 0.21 & 32.75 & 1.26 & 1.50 & 0.21 & 77.08 & 157.76 \\
0.05 & 0.01 & 0.35 & 0.01 & 2 & 2 & 0.5 & 100  & 0.1 & 0.21 & 32.71 & 1.26 & 1.56 & 0.21 & 77.09 & 157.56 \\
0.05 & 0.01 & 0.35 & 0.01 & 2 & 2 & 0.5 & 1000 & 0.1 & 0.21 & 32.76 & 1.26 & 1.46 & 0.21 & 77.09 & 157.33 \\
0.05 & 0.01 & 0.35 & 0.01 & 2 & 2 & 0.5 & 400  & 0.1 & 0.21 & 32.75 & 1.26 & 1.51 & 0.21 & 77.09 & 157.30 \\
0.05 & 0.01 & 0.35 & 0.01 & 1 & 1 & 0.5 & 1000 & 0.1 & 0.21 & 32.75 & 1.26 & 1.47 & 0.21 & 77.09 & 157.10 \\
0.05 & 0.01 & 0.35 & 0.01 & 2 & 2 & 0.5 & 700  & 0.1 & 0.21 & 32.75 & 1.26 & 1.47 & 0.21 & 77.09 & 156.08 \\
\hline
\end{tabular}
\end{table}

For these ten instances, \(\bar V^{\mathrm{my}}/\bar V^{\mathrm{Wh}}\)
is approximately \(0.67\), so the myopic rule loses roughly one third of 
the Whittle reward in this most adverse regime. By contrast, the forced-capacity
Lagrangian index policy is essentially tied with Whittle's policy on all ten cases, further
supporting its role as a strong VFA/Lagrangian benchmark.

\subsection{Breakdowns by capacity, reward profile, and population law}
\label{supp:numerics-breakdowns}

Tables~\ref{tab:S-by-capacity-exp1}--\ref{tab:S-by-capacity-exp3} report, for each experiment family, mean reward ratios \(\bar V^{\mathrm{Wh}}/\bar V^{\mathrm{my}}\) by capacity ratio \(M/N\). These tables make explicit the pattern summarized in the main text: the Whittle advantage is largest at \(M/N=0.05\), smaller at \(M/N=0.10\), and smaller still at \(M/N=0.20\).

For Experiment~1, Table~\ref{tab:S-by-rprofile-exp1} further breaks down results by the reward profile \((r^A,r^B)\). The Whittle/myopic reward ratio is smallest when the more adherence-persistent type receives the larger reward weight. Recall that the myopic rule assigns priority \(r_n x\), the immediate expected reward gain from activation, whereas Whittle's policy assigns priority \(m_n(x)\), which accounts for future adherence dynamics. Thus, the smaller ratio for \((r^A,r^B)=(2,1)\) suggests that giving the larger reward to the more adherence-persistent type makes the myopic priority closer to the dynamic index priority. In contrast, when \(r^A=r^B\) or the more lapse-prone type receives the larger reward weight, the Whittle advantage remains large.

For Experiment~2, Table~\ref{tab:S-by-composition-exp2} reports the same summaries by three-type population composition, while for Experiment~3, Table~\ref{tab:S-by-mixture-exp3} separates the baseline and stress mixture laws.

\begin{table}[!htbp]
\centering
\caption{Mean reward ratio \(\bar V^{\mathrm{Wh}}/\bar V^{\mathrm{my}}\) by capacity ratio \(M/N\) in Experiment~1.}
\label{tab:S-by-capacity-exp1}
\begin{tabular}{cccc}
\hline
\(M/N\) & 0.05 & 0.10 & 0.20 \\
\hline
mean \(\bar V^{\mathrm{Wh}}/\bar V^{\mathrm{my}}\) & 1.193 & 1.132 & 1.047 \\
\hline
\end{tabular}
\end{table}

\begin{table}[!htbp]
\centering
\caption{Mean reward ratio \(\bar V^{\mathrm{Wh}}/\bar V^{\mathrm{my}}\) by capacity ratio \(M/N\) in Experiment~2.}
\label{tab:S-by-capacity-exp2}
\begin{tabular}{cccc}
\hline
\(M/N\) & 0.05 & 0.10 & 0.20 \\
\hline
mean \(\bar V^{\mathrm{Wh}}/\bar V^{\mathrm{my}}\) & 1.211 & 1.163 & 1.063 \\
\hline
\end{tabular}
\end{table}

\begin{table}[!htbp]
\centering
\caption{Mean reward ratio \(\bar V^{\mathrm{Wh}}/\bar V^{\mathrm{my}}\) by capacity ratio \(M/N\) in Experiment~3.}
\label{tab:S-by-capacity-exp3}
\begin{tabular}{cccc}
\hline
\(M/N\) & 0.05 & 0.10 & 0.20 \\
\hline
mean \(\bar V^{\mathrm{Wh}}/\bar V^{\mathrm{my}}\) & 1.132 & 1.110 & 1.048 \\
\hline
\end{tabular}
\end{table}

\begin{table}[!htbp]
\centering
\caption{Mean reward ratio \(\bar V^{\mathrm{Wh}}/\bar V^{\mathrm{my}}\) by reward profile \((r^A,r^B)\) in Experiment~1.}
\label{tab:S-by-rprofile-exp1}
\begin{tabular}{ccccc}
\hline
\((r^A,r^B)\) & \((1,1)\) & \((1,2)\) & \((2,1)\) & \((2,2)\) \\
\hline
mean \(\bar V^{\mathrm{Wh}}/\bar V^{\mathrm{my}}\) & 1.164 & 1.149 & 1.021 & 1.164 \\
\hline
\end{tabular}
\end{table}

\begin{table}[!htbp]
\centering
\caption{Mean reward ratio \(\bar V^{\mathrm{Wh}}/\bar V^{\mathrm{my}}\) by population composition in Experiment~2.}
\label{tab:S-by-composition-exp2}
\begin{tabular}{lccc}
\hline
Composition \((\mathrm{prop}_A,\mathrm{prop}_B,\mathrm{prop}_C)\) & \((0.7,0.2,0.1)\) & \((0.6,0.3,0.1)\) & \((0.5,0.3,0.2)\) \\
\hline
mean \(\bar V^{\mathrm{Wh}}/\bar V^{\mathrm{my}}\) & 1.148 & 1.136 & 1.154 \\
\hline
\end{tabular}
\end{table}

\begin{table}[!htbp]
\centering
\caption{Mean reward ratio \(\bar V^{\mathrm{Wh}}/\bar V^{\mathrm{my}}\) by mixture law in Experiment~3.}
\label{tab:S-by-mixture-exp3}
\begin{tabular}{lcc}
\hline
Mixture law & baseline & stress \\
\hline
mean \(\bar V^{\mathrm{Wh}}/\bar V^{\mathrm{my}}\) & 1.087 & 1.106 \\
\hline
\end{tabular}
\end{table}

The composition and mixture-law breakdowns show that the Whittle policy's advantage is not
confined to a single population mix: it remains large across the three-type
compositions and is somewhat stronger under the stress mixture than under the
baseline mixture.

Overall, these supplementary results reinforce the main-text conclusions. The
Whittle policy is uniformly close to the Lagrangian bound, the forced-capacity
Lagrangian index policy is nearly indistinguishable from Whittle's, and the myopic rule can
be substantially misaligned in low-\(q\), lapse-prone stress regimes,
especially when capacity is tight.
}


\end{document}